\documentclass{amsart}
\usepackage[foot]{amsaddr} 
\usepackage[pagewise]{lineno}
\usepackage[margin=1.5in]{geometry}  
\usepackage{amsmath,amsthm,amssymb,graphicx,tikz,mathtools,enumitem,tikz-cd,tkz-graph,caption,subcaption,pdflscape,adjustbox,comment,amsfonts,float}
\usepackage{hyperref}
\hypersetup{colorlinks=true,allcolors=blue}
\newcommand{\bbB}{\mathbb{B}}
\newcommand{\bbC}{\mathbb{C}}

\newcommand{\bbN}{\mathbb{N}}

\newcommand{\bbQ}{\mathbb{Q}}
\newcommand{\bbR}{\mathbb{R}}
\newcommand{\bbS}{\mathbb{S}}

\newcommand{\bbZ}{\mathbb{Z}}

\newcommand{\calA}{\mathcal{A}}

\newcommand{\calF}{\mathcal{F}}

\newcommand{\calH}{\mathcal{H}}

\newcommand{\calK}{\mathcal{K}}
\newcommand{\calL}{\mathcal{L}}

\newcommand{\calR}{\mathcal{R}}

\newcommand{\frakS}{\mathfrak{S}}

\DeclareMathOperator{\It}{It}

\DeclareMathOperator{\PC}{PC}
\DeclareMathOperator{\WPC}{WPC}

\DeclareMathOperator{\GL}{GL}
\DeclareMathOperator{\h}{h}

\DeclarePairedDelimiter\floor{\lfloor}{\rfloor}

\DeclareMathOperator{\PA}{PA}

\DeclareMathOperator{\rot}{rot}

\DeclareMathOperator{\Pre}{Pre}
\DeclareMathOperator{\Fix}{Fix}
\DeclareMathOperator{\Mod}{Mod}

\DeclarePairedDelimiterX{\norm}[1]{\lVert}{\rVert}{#1}

\def\fillandplacepagenumber{%
 \par\pagestyle{empty}%
 \vbox to 0pt{\vss}\vfill
 \vbox to 0pt{\baselineskip0pt
   \hbox to\linewidth{\hss}%
   \baselineskip\footskip
   \hbox to\linewidth{%
     \hfil\thepage\hfil}\vss}}

\theoremstyle{plain}
\newtheorem{thm}{Theorem}[section]
\newtheorem{prop}[thm]{Proposition}

\newtheorem{cor}[thm]{Corollary}
\newtheorem{mainthm}{Theorem}

\newtheorem{fundlem}{Lemma}

\theoremstyle{definition}
\newtheorem{defn}[thm]{Definition}
\newtheorem{ex}[thm]{Example}
\newtheorem{rmk}[thm]{Remark}
\newtheorem{question}[thm]{Question}

\title{A Farey tree structure on a family of pseudo-Anosovs}
\author{Ethan Farber}
\email{farbereth@gmail.com}

\begin{document}
\newpage

\begin{abstract}
We introduce a new perspective on a procedure for generating pseudo-Anosov homeomorphisms from post-critically finite interval maps. The central idea is the realization of a tree structure on one such family of pseudo-Anosovs: individual systems, represented by the rational number measuring their rotation at infinity, are the vertices of the tree, while the edges encode dynamical relations between them. We also deepen the dictionary between one-, two-, and three-dimensional invariants associated with these systems.
\end{abstract}

\maketitle

\section{Introduction}

The goal of this paper is to explore the relationship between (1) uniformly expanding continuous maps of the unit interval, and (2) pseudo-Anosov homeomorphisms on a punctured sphere, using the theory of train tracks. Explicitly, the principal objects of study are:

\begin{enumerate}
\item a piecewise-linear map $f: [0,1] \to [0,1]$ whose postcritical orbits comprise a set of cardinality $p$, such that the slope of $f$ is always $\pm \lambda$, for some $\lambda > 1$; and
\item a pseudo-Anosov $\psi: S_{0, p+1} \to S_{0, p+1}$ with stretch factor $\lambda > 1$.
\end{enumerate}

\noindent The map $f$ is a \textit{PCF $\lambda$-expander}, and $\lambda$ is the \textit{stretch factor} of $f$. The double use of $\lambda$ is intentional: both systems have topological entropy $\log{\lambda}$, and we will be concerned with producing from a PCF $\lambda$-expander $f$ a pseudo-Anosov $\psi_f$ with stretch factor $\lambda$, via a procedure called \textit{thickening}. The singularities of $\psi_f$ coincide with the punctures of $S_{0, p+1}$, and are as follows:

\begin{itemize}
\item a punctured $1$-prong singularity for each of the $p$ postcritical points of $f$; and
\item a punctured $(p-2)$-prong singularity, necessarily fixed by $\psi_f$ and denoted $\infty$.
\end{itemize}

It is not always possible to thicken a given PCF $\lambda$-expander to a pseudo-Anosov. In \cite{F} we showed exactly when this is possible for a specific class of interval maps, called \textit{zig-zag maps}. Let $\PA(m)$ denote the set of zig-zag maps with $m \geq 2$ critical points that thicken to a pseudo-Anosov. By analyzing the orbit of the point $x=1$ under $f \in \PA(m)$, we defined an invariant $\Phi: \PA(m) \to \bbQ \cap (0,1)$ and proved that this invariant bijectively parameterizes $\PA(m)$ (cf. Theorem 2 in \cite{F} and Lemma \ref{fundlem:Rational} below).

Our first main result presents a geometric interpretation of this parameterization. Given a pseudo-Anosov $\psi$ fixing a $b$-prong singularity $q$, the \textit{rotation number} of $\psi$ at $q$ is $\rot_q(\psi)= \frac{a}{b}$ if $\psi$ rotates the prongs at $q$ by $a$ spaces counterclockwise.

\begin{mainthm}\label{mainthm:RotationNumber}
Fix $f \in \PA(m)$. Then $\Phi(f)$ measures the local clockwise rotation of $\psi_f$ at the fixed singularity $\infty$: that is,

\[
\Phi(f) = 1 - \rot_\infty(\psi_f).
\]
\end{mainthm}

Andr\'e de Carvalho suggested to us the possibility of Theorem \ref{mainthm:RotationNumber} in the context of his work with Toby Hall on generating pseudo-Anosovs from \textit{unimodal maps}, i.e. PCF $\lambda$-expanders with a single critical point. Indeed, we would be remiss to not mention that our techniques in \cite{F} and in the present paper take great inspiration from those of de Carvalho and Hall in \cite{dCH}. We prove Theorem \ref{mainthm:RotationNumber} in Section \ref{sec:rotation}. While we are not currently aware of a proof of the analogous statement in the case of unimodal maps, we suspect that the statement holds and is amenable to a similar proof strategy.

\bigskip

Let us further contextualize Theorem \ref{mainthm:RotationNumber}. Thickening an interval map $f: I \to I$ to a pseudo-Anosov $\psi_f: S \to S$ involves embedding $I$ into $S$ as an invariant train track $\tau$ for $\psi_f$. To do so, we add a loop of length $1$ to $I$ at each point in the forward orbit of a critical point of $f$, representing the fact that some iterate of $f$ makes a turn at that point. We add a puncture to $S$ within each loop to make it non-trivial in homology. In the language of train tracks, these loops are \textit{infinitesimal edges}. The remaining edges of $\tau$ (i.e., the subintervals of $I$) are \textit{expanding edges}, and produce Markov partitions for both $f$ and $\psi_f$. In particular, $f$ and $\psi_f$ are both Bernoulli processes with the same transition matrix, and so the dynamics of $f$ determine the dynamics of $\psi_f$ almost entirely.

Thus, one may hope to transport the techniques of one-dimensional dynamics to study the pseudo-Anosovs that are thickenings of interval maps. More specifically, one desires to interpret invariants of $f$ using invariants of $\psi_f$. Theorem \ref{mainthm:RotationNumber} is a first step in this direction. Indeed, while $\Phi(f)$ is initially defined using the permutation action of $f$ on the orbit of $x=1$, we see that it has a new interpretation via the dynamics of $\psi_f$.

Additionally, it is often fruitful to study dynamical systems as members of a larger family. Interval maps come in natural families determined by variations in their \textit{kneading data}. The maps in $\PA(m)$ constitute one such family (cf. Section \ref{subsection:Knead}). If the interval maps in a given family thicken to pseudo-Anosovs, it is natural then to study the pseudo-Anosovs as a collection. This perspective provides an alternative to the standard ways of creating families of pseudo-Anosovs, such as

\begin{enumerate}
\item composing Dehn twists along a given set of curves according to a pattern, or
\item taking branched covers of pseudo-Anosovs.
\end{enumerate}

Indeed, in case (1) the pseudo-Anosovs all act on the same surface $S$, while in case (2) the pseudo-Anosovs all have the same stretch factor. Varying the kneading data of the interval map generating a pseudo-Anosov, however, will generally change both the stretch factor of the pseudo-Anosov and the number of punctures of the surface on which it acts.\\

In this context, Theorem \ref{mainthm:RotationNumber} proposes analyzing the pseudo-Anosovs generated by $\PA(m)$ using $\bbQ \cap (0,1)$. As we will see, this is the vertex set of a special graph $\calF$ whose edges encode relations between pairs of maps in $\PA(m)$. We refer to $\calF$ as the \textit{Farey tree}, and in order to understand its structure we will make use of a second invariant of a map $f \in \PA(m)$: the \textit{digit polynomial} $D_f(t)$. The digit polynomial has integral coefficients and satisfies $D_f(\lambda(f)) = 0$. Its coefficients encode the orbit of $x=1$ under $f$, which is enough to determine the entire kneading data of $f$ due to the simplicity of zig-zag maps. 


By studying how $D_f$ depends on the location of $\Phi(f)$ in $\calF$, we obtain the following monotonicity result. Recall that $\lambda(f) = \lambda(\psi_f)$.

\begin{mainthm}\label{mainthm:mono}
Given $f, g \in \PA(m)$, we have

\[
\Phi(f) < \Phi(g) \iff \lambda(\psi_f) < \lambda(\psi_g).
\]
\end{mainthm}

Since the topological entropy of both $f$ and $\psi_f$ is $\log{\lambda(\psi_f)}$, one may interpret Theorem \ref{mainthm:mono} as a monotonicity result for the entropy of the one-paramater family of pseudo-Anosovs generated by $\PA(m)$. Entropic monotonicity results exist for many natural families of interval maps (cf. for example \cite{MTr} and \cite{BvS}).

The proof of Theorem \ref{mainthm:mono} is a chain of monotonicity statements. The key driver is Theorem \ref{thm:Farey_transform}, which presents a pair of transformation laws for the kneading data of $f = \Phi^{-1}(q)$ as $q \in \calF$ varies. We treat $\calF$ in Section \ref{sec:Farey} and describe how it models both $\PA(m)$ and $\Pi(m)$ in Section \ref{sec:model}. We then prove Theorem \ref{mainthm:mono} in Section \ref{sec:mono}.\\

Techniques from kneading theory also allow us to obtain best possible bounds on $\lambda(\psi_f)$.

\begin{cor}
Let $f \in \PA(m)$.  Then,

\[
\inf \{\lambda(\psi_f) : f \in \PA(m)\} = \frac{m+1 + \sqrt{(m+1)^2 - 8}}{2}
\]

\noindent and

\[
\sup \{\lambda(\psi_f) : f \in \PA(m)\} = m+1.
\]

\end{cor}

\begin{proof}
Given $f \in \PA(m)$, $\lambda(f)$ is equal to the magnitude of the slope of the piecewise-linear map $f: I \to I$. The integer floor of $\lambda(f)$ is $\floor{\lambda(f)}=m$. Thus, $\lambda(f)$ is bounded from below by $m$ and from above by $m+1$.

For $n \geq 1$, define the maps $g_n, h_n \in \PA(m)$ by

\[
g_n = \Phi^{-1} \left ( \frac{1}{n} \right )\hspace{5mm} \text{and} \hspace{5mm} h_n = \Phi^{-1} \left ( \frac{n-1}{n} \right ).
\]

\noindent Theorem \ref{mainthm:mono} implies that the sequences $\{\lambda(g_n)\}_n$ and $\{\lambda(h_n)\}_n$ are monotonically decreasing and monotonically increasing, respectively. Since these sequences are bounded, it follows that they converge. Indeed, we must have

\[
\lambda_0 := \lim_{n \to \infty} \lambda(g_n) = \inf\{\lambda(f) : f \in \PA(m)\}
\]

\noindent and

\[
\lambda_1 := \lim_{n \to \infty} \lambda(h_n) = \sup\{\lambda(f): f \in \PA(m) \}.
\]

The maps $g_n$ converge uniformly to a $\lambda_0$-zig-zag map $\textbf{0}_m$, and the maps $h_n$ converge uniformly to a $\lambda_1$-zig-zag map $\textbf{1}_m$ (cf. Definition \ref{def:0and1}). In the case of $\textbf{1}_m$, it is not hard to see from the principal kneading sequence (cf. Definition \ref{def:prin}) of $h_n$ that $\textbf{1}_m(1) = \lambda_1-m$ is a fixed point of $\textbf{1}_m$ in the subinterval $[m \cdot \lambda_1^{-1}, 1]$, on which $\textbf{1}_m$ is defined by the linear map

\[
\textbf{1}_m(x) = \lambda_1 x - m \hspace{5mm} \text{for $x \in [m \cdot \lambda_1^{-1}, 1]$}.
\]

\noindent From this one determines that $\lambda_1=m+1$. On the other hand, from the principal kneading sequence of $g_n$ one sees that $\textbf{0}_m(1) = \lambda_0-m$ is a fixed point of $\textbf{0}_m$ in the subinterval $[(m-2) \cdot \lambda_0^{-1}, (m-1) \cdot \lambda_0^{-1}]$. On this subinterval $\textbf{0}_m$ is defined by the linear map

\[
\textbf{0}_m(x) = \lambda_0 x - (m-2) \hspace{5mm} \text{for $x \in [(m-2) \cdot \lambda_0^{-1}, (m-1) \cdot \lambda_0^{-1}]$}.
\]

\noindent Thus we find 

\[
\lambda_0-m = \textbf{0}_m(\lambda_0 -m) = \lambda_0^2 - m\lambda_0 - (m-2).
\]

\noindent Solving for $\lambda_0$ gives

\[
\lambda_0 = \frac{m+1 + \sqrt{(m+1)^2-8}}{2}.
\]
\end{proof}

Just as Theorem \ref{mainthm:RotationNumber} provides an interpretation of the first invariant $\Phi(f)$ from the perspective of $\psi_f$, we seek to understand the second invariant $D_f$ using two- and three-dimensional techniques. Such an understanding provides immediate leverage over several topological invariants associated to $\psi_f$. Indeed, we will see that $D_f$ will coincide or nearly coincide with these invariants.

The core idea in the following theorem is that, despite its original definition, $D_f$ is the characteristic polynomial of several matrices naturally associated to $f$ or to $\psi_f$.

\begin{mainthm}\label{mainthm:poly_invts}
Let $f \in \PA(m)$ and set $n = |\PC(f)| = 1+\deg(D_f)$. Set $\lambda = \lambda(f) = \lambda(\psi_f)$. 

\begin{enumerate}
\item The digit polynomial determines the Artin-Mazur zeta function of $f$: 

\[
\zeta_f(t) = \frac{1}{\calR(D_f(t))} = \frac{1}{\calR(\det(tI-W_f))}.
\]

\item The digit polynomial is equal to the strong Markov polynomial of $f$:

\[
D_f(t) = \chi_M(f; t).
\]

\item The digit polynomial determines the homology, symplectic, and puncture polynomials of $\psi_f$, defined by Birman-Brinkmann-Kawamuro (\cite{BBK}):

\[
D_f(t) = h(\psi_f; t) = \begin{cases}
s(\psi_f; t) & \text{if $n$ is odd}\\
s(\psi_f; t)(t+1) & \text{if $n$ is even.}
\end{cases}
\]

\item Let $\bbS$ be the surface obtained from the orientation double cover of $\psi_f$ by filling in the lifts of the punctured $1$-prong singularities of $\psi_f$. Denote by $\chi_+(t)$ (resp., $\chi_-(t)$) the characteristic polynomial of the lift $\psi_+$ (resp., $\psi_-$) acting on $H_1(\bbS; \bbZ)$. Then

\[
D_f(t) = \chi_+(t) \hspace{5mm} \text{and} \hspace{5mm} D_f(-t) = \chi_-(t).
\]


\item Let $\beta_f$ be any $n$-braid representative of $\psi_f$ obtained by ripping open $\infty \in S_{0,n+1}$ to a boundary circle. Let $\bbB(\beta_f, z)$ denote the reduced Burau matrix for $\beta_f$, and set $\chi(\beta_f; t) = \det(tI-\bbB(\beta_f, -1))$. Then

\[
\chi(\beta_f; t) = \begin{cases}
D_f(t) & \text{if $\chi(\beta_f; t)$ has $\lambda$ as a root}\\
D_f(-t) & \text{if $\chi(\beta_f; t)$ has $-\lambda$ as a root.}
\end{cases}
\]

\noindent Moreover, composing $\beta_f$ with the full twist $\Delta_n^2$ negates the sign of the variable $t$.
\end{enumerate}
\end{mainthm}

Note that since $D_f(t)$ is reciprocal, i.e. $\calR(D_f(t)) = D_f(t)$, statement (\ref{poly_invts:zeta}) implies that 

\[
\zeta_f(t) = \frac{1}{D_f(t)}.
\]


We remark that any pseudo-Anosov that is the thickening of a PCF $\lambda$-expander is defined on a punctured sphere and has at most one singularity with $3$ or more prongs. In a forthcoming paper with Karl Winsor we prove the converse statement: every such pseudo-Anosov is the thickening of a PCF $\lambda$-expander $f$. In this more general setting, the digit polynomial is replaced by the strong Markov polynomial $\chi_M(f; t)$, which is simply the characteristic polynomial of a minimal Markov partition associated to $f$. Minor modifications to the proof of Theorem \ref{mainthm:poly_invts} will then prove analogous statements for the relevant invariants of $\psi_f$.

We will close with a brief discussion of these and other thoughts in Section 8, after proving Theorem \ref{mainthm:poly_invts} in Section \ref{sec:poly_invts}.

\subsection*{Acknowledgements}
First and foremost, the author would like to thank his advisor Kathryn Lindsey, for her ever-keen insights, feedback, and encouragement. A great deal of gratitude goes to Andr\'e de Carvalho and Toby Hall, whose ingenious and precise work was the well to which the author repeatedly returned over many months. Chenxi Wu was the source of innumerable helpful conversations.

And finally, it cannot be overstated how deeply the author relied on the camaraderie, empathy, and compassion of the graduate students at Boston College. We still do not have a contract.

\setcounter{tocdepth}{1}
{ \hypersetup{linkcolor=black} \tableofcontents }


\section{Preliminaries}

In this section we recall the key definitions and background from \cite{F} needed to prove Theorem \ref{mainthm:RotationNumber}. We direct the reader to Sections 2, 3, 5, 6, and 7 of \cite{F} for details. The informed reader may skip this section.

\subsection{The sets $\PA(m)$}

Throughout this paper and unless stated otherwise, we will use $I$ to denote the closed unit interval. All maps $f: I \to I$ are assumed to be continuous with finitely many \textit{critical points}, i.e. points having no neighborhood on which $f$ is injective. A \textit{critical value} is the image of a critical point.

\begin{defn}
Choose $\lambda > 1$. A \textit{uniform $\lambda$-expander} is a piecewise-linear map $f: I \to I$ such that each linear branch of $f$ has slope $\pm \lambda$. In this case, we say $f$ has \textit{constant slope $\lambda$}.
\end{defn}

\begin{rmk}
Misiurewicz and Szlenk \cite{MS} proved that the topological entropy of any uniform $\lambda$-expander $f$ is

\[
\h(f) = \log{\lambda}.
\]
\end{rmk}

\begin{defn}\label{defn:Zigzag}
A \textit{$\lambda$-zig-zag map} (or simply \textit{zig-zag map} or \textit{zig-zag} when $\lambda$ is understood) is a uniform $\lambda$-expander $f: I \to I$ whose critical points are precisely the numbers 

\[
k_i = i \cdot \lambda^{-1}, \hspace{5mm} \text{for $i = 1, \ldots, \floor{\lambda}$.}
\]

\noindent Note that a zig-zag $f$ satisfies $f(0) \in \{0, 1\}$. We say $f$ is \textit{positive} if $f(0)=0$, and \textit{negative} if $f(0)=1$.
\end{defn}

Note that the number of critical points of a $\lambda$-zig-zag map $f$ is $m=\floor{\lambda}$. Thus when $\lambda \geq 2$, the critical values of $f$ are precisely $0$ and $1$. Throughout the paper, fixing $m \geq 2$ amounts to fixing a number of critical points for the underlying zig-zag map, which is necessary in order to properly describe the combinatorics.

\begin{defn}\label{defn:postcrit}
A \textit{postcritical point} of $f: I \to I$ is a point of the form $f^n(k)$ for $n \geq 1$ and $k$ a critical point. The \textit{postcritical set} of $f$ is the set of all postcritical points of $f$, and is denoted $\PC(f)$. If $\PC(f)$ is a finite set, then $f$ is \textit{postcritically finite} or \textit{PCF}.
\end{defn}

\subsection{Pseudo-Anosovs, train tracks, and thickening}

Let $S = S_{g,p}$ be the closed, connected surface of genus $g$ with $p$ marked points, which we will often treat as punctures. Assume that $6g-6+2p>0$.

\begin{defn}
A \textit{pseudo-Anosov} is a homeomorphism $\psi: S \to S$ for which there exist (1) two transverse singular measured foliations $(\calF^s, \mu^s)$ and $(\calF^u, \mu^u)$ of $S$ and (2) a real number $\lambda > 1$, such that:

\begin{itemize}
\item $\psi \cdot (\calF^s, \mu^s) = (\calF^s, \lambda^{-1}\mu^s)$, and
\item $\psi \cdot (\calF^u, \mu^u) = (\calF^u, \lambda \mu^u)$.
\end{itemize}

\noindent That is, $\psi$ preserves both $\calF^s$ and $\calF^u$, while scaling the former's transverse measure by $\lambda^{-1}$ and the latter's by $\lambda$.
\end{defn}

Call a point $s \in \overline{S}$ a \textit{$k$-pronged point} if $k=k(s)$ half-leaves of $\calF^u$ land at $s$. If $k \neq 2$, we call $s$ a \textit{$k$-pronged singularity} of $\psi$. Sometimes $2$-pronged points are considered degenerate singularities. There are only finitely-many (non-degenerate) singularities of $\psi$, and the Euler-Poincar\'e formula imposes restrictions on the number of prongs at these singularities:

\begin{equation}\label{eq:EP}
2 \chi \Big (\overline{S} \Big ) = \sum_s 2-k(s).
\end{equation}

\noindent Here the sum is over all singularities of $\psi$. Note that only $1$-pronged singularities lend a positive contribution to the above sum. For technical reasons, these singularities only occur at punctures of $S$.

We refer to $\calF^s$ and $\calF^u$ as the \textit{stable} and \textit{unstable foliations} of $\psi$, respectively. This naming convention comes from the action of $\psi$ on the space $\mathcal{PMF}$ of projective measured foliations of $S$: the classes of $\calF^s$ and $\calF^u$ are the two fixed points of $\psi$, with the former being the unique sink and the latter the unique source. See \cite{Th3} for more.

Suppose that $\psi$ is orientation-preserving. Away from its singularities, one may think of $\psi$ as acting in local coordinates by an affine map having derivative

\[
D\psi = \begin{pmatrix}
\lambda & 0 \\
0 & \lambda^{-1}
\end{pmatrix}.
\]

\noindent Here the first coordinate is in the direction of $\calF^u$, and the second in the direction of $\calF^s$. The number $\lambda > 1$, called the \textit{stretch factor} or \textit{dilatation} of $\psi$, is a fundamental invariant. As it turns out, $\lambda$ is an algebraic integer of degree at most $6g-6+2p$, and the topological entropy of $\psi$ is $\h(\psi) = \log{\lambda}$. 

Any pseudo-Anosov $\psi$ may be reconstructed, along with its invariant foliations, by investigating its action on particular graphs embedded in $S$, called \textit{train tracks}.

\begin{defn}
A \textit{train track} is a graph $\tau$ smoothly embedded in $S$ satisfying the following properties:

\begin{enumerate}
\item For each vertex $v$ of $\tau$, each smooth path $p: [-1,1] \to \tau$ with $p(0)=v$ has the same tangent line at $v$. This tangency condition creates cusps between adjacent incoming edges with the same unit tangent vector at $v$.
\item Each connected component of $S \setminus \tau$ is either
\begin{enumerate}
\item \label{track_comp1} an un-punctured disc with at least three cusps on its boundary, or
\item \label{track_comp2} a once-punctured disc with at least one cusp on its boundary.
\end{enumerate}
\end{enumerate}
\end{defn}

If the image $\psi(\tau)$ is homotopic to $\tau$ in a restricted sense, then we say $\tau$ is an \textit{invariant train track} for $\psi$, and by concatenating $\psi$ with the homotopy we obtain a graph map $f: \tau \to \tau$. We call $f$ a \textit{train track map}, and it can be used to reconstruct $\psi$ along with the foliations $\calF^s$ and $\calF^u$. Indeed, the edges of $\tau$ form a 1-dimensional Markov partition for $f$. The subset of \textit{expanding edges} of $\tau$, meaning those edges $e$ such that $f^n(e) = \tau$ for large $n$, produce a 2-dimensional Markov partition for $\psi$ consisting of rectangles. See \cite{BH} for more.

\begin{rmk}
If $\tau$ is an invariant train track for $\psi$, then the complementary components of $S \setminus \tau$ determine the singularity data of $\psi$. Explicitly, $\psi$ has exactly one $k$-prong singularity for each component with $k$ cusps. Components of type \ref{track_comp1} correspond to unpunctured singularities, while components of type \ref{track_comp2} correspond to punctured singularities.
\end{rmk}

The dynamics of $f: \tau \to \tau$ thus determine the dynamics of $\psi: S \to S$, up to a null set. We are interested in producing pseudo-Anosovs $\psi$ from suitable graph maps $f$, and relating invariants from the two different regimes. 

In \cite{dC} Andr\'e de Carvalho developed a way of ``thickening" a PCF interval map $f$ to create a two-dimensional system, called a \textit{generalized pseudo-Anosov}. Via this thickening procedure, one finds a minimal smooth branched 1-submanifold $\tau \subseteq S^2$, called a \textit{generalized invariant train track}, which is invariant under the thickened map, up to pseudo-isotopy. When $\tau$ is a finite graph, the thickening procedure outputs a pseudo-Anosov $\psi_f: S^2 \to S^2$. In this case, $\tau$ is a train track in the classical sense, and can be thought of as the starting interval $I$ decorated with loops at each point of the postcritical set of $f$. The side of $I$ on which each loop lies is meaningful, and lends a combinatorial flavor to the question of when this $\tau$ could possibly be finite. Importantly, the \textit{train track map} on $\tau$ representing $\psi_f$ is essentially $f$, so the dynamics of $f$ determine the dynamics of $\psi_f$.

In Section 3 of \cite{F} the author showed that there is essentially only one way to thicken a zig-zag map $f: I \to I$ that could possibly result in a finite $\tau$, i.e. that could produce a pseudo-Anosov. If this procedure does produce a pseudo-Anosov from $f$, then we say that $f$ is of \textit{pseudo-Anosov type}, and we denote the pseudo-Anosov by $\psi_f$.

\begin{defn}\label{defn:PAB}
Fix $m \geq 2$. We define $\PA(m)$ to be the set of zig-zags $f$ of pseudo-Anosov type such that

\begin{enumerate}
\item $f$ has $m$ critical points, and
\item $|\PC(f)| \geq 4$.
\end{enumerate}

\noindent We also define $\Pi(m) = \{\psi_f : f \in \PA(m)\}$.
\end{defn}

\begin{rmk}
Maps $f \in \PA(m)$ are distinguished by their slopes $\lambda$, hence by their topological entropy $\log{\lambda}$. Since $\psi_f$ has the same entropy as $f$, it follows that the map $f \mapsto \psi_f$ is a bijection from $\PA(m)$ onto $\Pi(m)$.
\end{rmk}


\subsection{The permutations $\rho_m(n,k)$ and the map $\Phi$}\label{subsec:perm}

This subsection follows closely the discussion in Section 6 of \cite{F} We will define the permutations $\rho_m(n,k)$, which describe precisely the permutation types of the maps $f \in \PA(m)$ (cf. Theorem \ref{thm:PermType} below).

\begin{defn}
Suppose $f$ is a zig-zag map for which the orbit of $x=1$ is periodic of minimal period $n$. Order this orbit in increasing order: $x_1<x_2<\cdots<x_n=1$. The \textit{permutation type} of $f$ is then the permutation $\rho(f) \in \frakS_n$ such that $f(x_i) = x_{\rho(f)(i)}$.
\end{defn}

\begin{rmk}
If $f$ is a zig-zag of pseudo-Anosov type, then there is a bijective correspondence between $\PC(f)$ and the $1$-prong singularities of the invariant foliations of the pseudo-Anosov $\psi_f$. Since $\psi_f$ permutes these singularities, $f$ necessarily permutes the elements of $\PC(f)$, and hence $x=1$ is periodic.
\end{rmk}

We will now describe the permutation types of $f \in \PA(m)$. These have slightly different forms, depending on the following three cases for $m$:

\begin{enumerate}
\item $m \geq 4$ even, in which case $\rho(f) = \rho_e(n,k)$ for some $n$ and $k$,
\item $m \geq 3$ odd, in which case $\rho(f) = \rho_o(n,k)$, or
\item $m=2$, in which case $\rho(f) = \rho_2(n,k)$.
\end{enumerate}

\noindent See Theorem \ref{thm:PermType} below.

\begin{defn}
Given $n \geq 3$ and $2 \leq k \leq n-1$, define $\rho_e(n,k) \in \frakS_n$ to be the permutation such that

\[
\rho_e(n,k)(i)=
\begin{cases}
n & \text{if $i=1$}\\
i+(n-k) & \text{if $2 \leq i \leq k-1$}\\
i-(k-1) & \text{if $k \leq i \leq n$}
\end{cases}
\]

\noindent When $k=2$ we interpret this definition to mean

\begin{align*}
\rho_e(n,2)(i) & = i-1 \pmod{n}\\
& = i+(n-1) \pmod{n}.
\end{align*}

\end{defn}

\begin{rmk}\label{rmk:Transitive}
Here is an observation that will be important for the proof of Theorem \ref{mainthm:RotationNumber}. After deleting the symbol $1$ from the cycle decomposition of $\rho_e(n,k)$, one obtains a new permutation $\tilde{\rho} \in \frakS_{n-1}$. Shifting all labels down by $1$, one sees that this permutation is the rotation by $n-k$ modulo $n-1$:

\[
\tilde{\rho}(i) = i+(n-k) \pmod{n-1}
\]

\noindent In particular, we see that

\[
\text{$\rho_e(n,k)$ is an $n$-cycle $\iff \tilde{\rho}$ is an $(n-1)$-cycle $\iff \gcd(n-k, n-1)=1$.}
\]

\noindent This is the motivation behind the definition of the map $\Phi: \PA(m) \to \bbQ \cap (0,1)$ (cf. Definition \ref{defn:Rotation} below).
\end{rmk}

\begin{defn}
Given $n \geq 3$ and $2 \leq k \leq n-1$, define $\rho_o(n,k) \in \frakS_{n+1}$ to be the permutation such that

\[
\rho_o(n,k)(i)=\begin{cases}
n & \text{if $i=0$}\\
0 & \text{if $i=1$}\\
i+(n-k) & \text{if $2 \leq i \leq k-1$}\\
i-(k-1) & \text{if $k \leq i \leq n$}
\end{cases}
\]

\noindent In other words, $\rho_o(n,k)$ is the permutation on $n+1$ letters obtained by inserting the symbol $0$ in between $1$ and $n$ in the orbit of $n$ under the permutation $\rho_e(n,k)$.
\end{defn}

\begin{defn}
Let $n \geq 3$. For $2 \leq k \leq n-1$, set 

\[
\kappa(n,k) = (1, 2, \ldots, k-1)(k) \cdots (n) \in \frakS_n
\]

\noindent and define

\[
\rho_2(n,k) = [\kappa(n,k)]^{-1} \circ \rho_e(n,k) \circ \kappa(n,k).
\]
\end{defn}

We write $\rho_m(n,k)$ when $m$ is not specified: thus,

\[
\rho_m(n,k) = \begin{cases}
\rho_e(n,k) & \text{if $m \geq 4$ even}\\
\rho_o(n,k) & \text{if $m \geq 3$ odd}\\
\rho_2(n,k) & \text{if $m=2$}.
\end{cases}
\]

It follows from Remark \ref{rmk:Transitive} that for any $m \geq 2$,

\[
\text{$\rho_m(n,k)$ is a transitive permutation $\iff \gcd(n-k, n-1)=1$.}
\]

The following theorem summarizes Sections 5 and 6 of \cite{F} 

\begin{thm}\label{thm:PermType}
Let $f: I \to I$ be a $\lambda$-zig-zag map with $\floor{\lambda}=m \geq 2$. If $f \in \PA(m)$, then $\rho(f)=\rho_m(n,k)$ for some $n, k$ satisfying

\[
\gcd(n-k, n-1) = 1.
\]

\noindent Moreover, for each such $\rho_m(n, k)$ there is a unique $f \in \PA(m)$ such that $\rho(f)=\rho_m(n,k)$.
\end{thm}

We introduce the notation $\calF = \bbQ \cap (0,1)$.

\begin{defn}\label{defn:Rotation}
Fix $m \geq 2$. Define $\Phi: \PA(m) \to \calF$ to be the map

\[
\text{$\Phi(f) = \frac{n-k}{n-1}$ \ if $\rho(f)=\rho_m(n,k)$}.
\]

\noindent Note that $\Phi$ always outputs fractions in lowest terms.
\end{defn}

We are now in a position to state the following fundamental fact, which appeared as Theorem 2 in \cite{F}.

\begin{fundlem}\label{fundlem:Rational}
For each $m \geq 2$ the map $\Phi: \PA(m) \to \calF$ is a bijection.
\end{fundlem}

A central goal of the present paper is to argue that $\PA(m)$ in fact respects the combinatorial structure of a special graph $\calF$ whose vertex set is $\calF$. We call this graph the \textit{Farey tree}. See Section \ref{sec:Farey} for more.

\subsection{The digit polynomial $D_f$}\label{subsection:Digit}

Let $f \in \PA(m)$ for $m \geq 2$ even, and write $\Phi(f)=\frac{a}{b}$. The point $x=1$ is strictly periodic under $f$ of minimal period $b+1$. Denote the elements of this orbit by $x_i=f^i(1)$ for $i=0, \ldots, b+1$. Since $f$ is a $\lambda$-zig-zag map, the $x_i$ satisfy a collection of linear relations:

\begin{equation*}
x_i = a_i \pm \lambda x_{i-1}, \hspace{1cm} \text{for $i=1, \ldots, b+1$}
\end{equation*}

\noindent where the indices are taken modulo $b+1$ and each $a_i$ is an integer satisfying $|a_i| \leq m$. In fact, in Section 7 of \cite{F} the author proves the further restrictions

\begin{equation}\label{eqn:recur1}
x_i = \begin{cases}
\lambda x_{i-1} - m \ \text{or}\ \lambda x_{i-1} - (m-2) & \text{if $1 \leq i \leq b-1$}\\
m - \lambda x_{i-1} & \text{if $i=b$}\\
2 - \lambda x_{i-1} & \text{if $i=b+1$.}
\end{cases}
\end{equation}

Since $x_0=x_{b+1}=1$, composing these relations and solving for $0$ produces a monic integral polynomial relation in $\lambda$, of the form

\begin{equation}\label{eqn:digitpoly}
0=D_f(\lambda) = \lambda^{b+1} +1 -\sum_{i=1}^b c_i \lambda^{b+1-i}.
\end{equation}

\noindent One checks that the $c_i$ are integers satisfying

\begin{equation}\label{eqn:coefrelations}
c_i=\begin{cases}
\lambda x_{i-1} - x_i & \text{if $1 \leq i \leq b-1$}\\
\lambda x_{i-1} + x_i & \text{if $i=b$}
\end{cases}.
\end{equation}

If instead $m$ is odd and $\Phi(f)=\frac{a}{b}$ then $x=1$ is periodic of minimal period $b+2$. With notation as before, the corresponding restrictions are similar:

\begin{equation}\label{eqn:recur2}
x_i=\begin{cases}
\lambda x_{i-1} - m \ \text{or} \ \lambda x_{i-1} - (m-2) & \text{if $1 \leq i \leq b-1$}\\
m - \lambda x_{i-1} & \text{if $i=b$}\\
\lambda x_{i-1} -1 & \text{if $i=b+1$}\\
1-\lambda x_{i-1} & \text{if $i=b+2$.}
\end{cases}
\end{equation}

\noindent Composing these relations and solving for $0$ produces the equation

\[
0=\lambda D_f(\lambda),
\]

\noindent where $D_f(\lambda)$ has the same form as in (\ref{eqn:digitpoly}) and the $c_i$ satisfy (\ref{eqn:coefrelations}).

In either case, we make the following definition.

\begin{defn}\label{defn:DigitPoly}
The polynomial $D_f(t) \in \bbZ[t]$ is the \textit{digit polynomial} of $f$.
\end{defn}

By definition, $\lambda$ is a root of $D_f$. Less obvious is the fact that $D_f$ is the characteristic polynomial of a matrix naturally associated to $f$, called the \textit{strong Markov matrix}. This fact is crucial to the proof of Theorem \ref{mainthm:poly_invts}. See Sections \ref{subsec:Markov} and \ref{subsec:zeta}.

\subsection{A little kneading theory}\label{subsection:Knead}

Let us recall a few basic definitions from kneading theory. For the moment, let $f: I \to I$ be any PCF piecewise monotone map with finitely many critical points $k_1, \ldots, k_m$. We introduce the notation

\[
\text{$I_0=[0,k_1)$, \ $I_m=(k_m,1]$, and $I_j=(k_j,k_{j+1})$ for $j=1, \ldots, m-1$.}
\]

\noindent These are the intervals of monotonicity for $f$. Put $\calA=\calA_m = \{0, k_1, 1, k_2, 2, \ldots, k_m, m\}$, an alphabet on $2m+1$ letters.

\begin{defn}
For any $x \in I$, the \textit{address} of $x$ is

\[
A(x) = \begin{cases}
j & \text{if $x \in I_j$}\\
k_j & \text{if $x=k_j$}
\end{cases}
\]

\noindent The \textit{itinerary} of $x$ is the sequence

\[
\It_f(x) = (A(x), A(f(x)), A(f^2(x)), \ldots).
\]

\noindent Given a critical point $k_j$, its image $f(k_j)$ is a \textit{critical value}. The \textit{kneading sequences} of $f$ are the itineraries of these critical values:

\[
\text{$\calK_j(f) = \It_f(f(k_j))$ \ for $j=1, \ldots, m$.} 
\]

\noindent The \textit{kneading data} of $f$ is the vector

\[
\calK_f = (\calK_1(f), \ldots, \calK_m(f)).
\]
\end{defn}

\begin{rmk}\label{rmk:kneaddata}
In the case that $f \in \PA(m)$, the kneading data $\calK_f$ takes a very particular form. From Corollary 5.6 and Definition 5.8 of \cite{F} we know that $f(0)=0$ if $m$ is even and $f(0)=1$ if $m$ is odd. Since the only critical values of $f$ are $0$ and $1$ we can say even more: if $m$ is even, then

\[
\calK_j(f) = \begin{cases}
\It_f(1) & \text{if $j$ is odd}\\
0^\infty & \text{if $j$ is even,}
\end{cases}
\]

\noindent while if $m$ is odd we have

\[
\calK_j(f)=\begin{cases}
0 \cdot \It_f(1) & \text{if $j$ is odd}\\
\It_f(1) & \text{if $j$ is even.}
\end{cases}
\]

In each case $\calK_f$ is completely determined by $\It_f(1)$, which is strictly periodic since $f \in \PA(m)$. For this reason, we make the following definition.

\end{rmk}

The \textit{$l$-th prefix} of a word $A=(A_0, A_1, \ldots)$ of length at least $l$ is the subword 

\[
\Pre_l(A) = (A_0, \ldots, A_{l-1}).
\]

\begin{defn}\label{def:prin}
Let $f \in \PA(m)$ and write $\Phi(f)=\frac{a}{b}$. The \textit{principal kneading sequence} $\nu(f)$ of $f$ is the first period of $\It_f(1)$:

\[
\nu(f) = \Pre_{b(m)}(\It_f(1)),
\]

\noindent where

\[
b(m) = \text{the minimal period of $1$ under $f$} = \begin{cases}
b+1 & \text{if $m$ is even}\\
b+2 & \text{if $m$ is odd.}
\end{cases}
\]

\noindent As with $\It_f(1)$, the sequence $\nu(f)$ is 0-indexed: for $0 \leq i \leq b(m)$,

\[
\nu_i(f) = (\It_f(1))_i = A(f^i(x)).
\]
\end{defn}

\bigskip

The next proposition is immediate from Equations (\ref{eqn:recur1}) and (\ref{eqn:recur2}), and relates the coefficients of $D_f$ to the entries of $\nu(f)$.

\begin{prop}\label{prop:knead}
Let $f \in \PA(m)$ with digit polynomial

\[
D_f(t) = t^{b+1}+1 - \sum_{i=1}^b c_it^{b+1-i}.
\]

\noindent Then the principal kneading sequence of $f$ is given by

\begin{equation}
\nu_i(f) = \begin{cases}
c_{i+1} & \text{if $0 \leq i \leq b-2$}\\
m-1 & \text{if $i=b-1$}\\
k_1 & \text{if $i=b$}\\
0 & \text{if $m$ is odd and $i=b+1$.}
\end{cases}
\end{equation}

\end{prop}

The next statement shows how to compute the coefficients of $D_f$ independently of $\nu(f)$. It is Theorem 3 in \cite{F}, and will be fundamental to our work in Sections \ref{sec:model} and \ref{sec:mono}.

\begin{fundlem}\label{fundlem:Digit}
Let $f \in \PA(m)$ with $\Phi(f)=\frac{a}{b}$. Define $L: [0,b] \to \bbR$ by $L(t)=\frac{a}{b} \cdot t$. Then

\[
D_f(t)=t^{b+1}+1-\sum_{i=1}^b c_it^{b+1-i},
\]

\noindent where the $c_i$ satisfy

\begin{equation}
c_i=\begin{cases}
m & \text{if $L(t) \in \bbN$ for some $t \in [i-1,i]$}\\
m-2 & \text{otherwise.}
\end{cases}
\end{equation} 

\noindent In particular, $c_i=c_{b-i}$, so $D_f$ is reciprocal: that is,

\[
D_f(t)=t^{b+1} D_f(t^{-1}).
\]
\end{fundlem}

The previous two statements imply the following simple description of the principle kneading sequence. We preserve notation.

\begin{cor}\label{cor:knead}
The principal kneading sequence of $f \in \PA(m)$ satisfies

\begin{equation}
\nu_i(f) = \begin{cases}
m & \text{if $0 \leq i \leq b-2$ and $L(t) \in \bbN$ for some $t \in [i, i+1]$}\\
m-2 & \text{if $0 \leq i \leq b-2$ and $L(t) \not \in \bbN$ for some $t \in [i, i+1]$}\\
m-1 & \text{if $i=b-1$}\\
k_1 & \text{if $i=b$}\\
0 & \text{if $m$ is odd and $i=b+1$.}
\end{cases}
\end{equation}
\end{cor}

As these statements show, the following invariants of $f \in \PA(m)$ determine each other:

\begin{enumerate}
\item the fraction $\Phi(f) = \frac{a}{b}$,
\item the polynomial $D_f(t)$, and
\item the principal kneading sequence $\nu(f)$. 
\end{enumerate}

\section{The proof of Theorem \ref{mainthm:RotationNumber}}\label{sec:rotation}

As we saw in Section \ref{subsec:perm}, the definition of the map $\Phi$ assigning to the interval map $f \in \PA(m)$ a fraction in $\calF = \bbQ \cap (0,1)$ is opaque. In this section, we prove that $\Phi(f)$ has an interpretation from the perspective of the pseudo-Anosov $\psi_f$.

\addtocounter{mainthm}{-3}

\begin{mainthm}
Suppose $f \in \PA(m)$. Then

\[
\Phi(f) = 1- \rot_\infty(\psi_f).
\]
\end{mainthm}

Before beginning the proof, we set some notation. 

In \cite{F} the author constructed the \textit{Galois lift} of $f$, following W. Thurston's example in \cite{Th2}. This is the piecewise-affine map $f_G: I \times \bbR \to I \times \bbR$ defined by

\[
f_G(x,y) = \left ( f_j(x), \tilde{f}_j(y) \right ) \ \text{if $x \in I_j$.}
\]

\noindent Here, $f_j(x) = a_j \pm \lambda x$ is the linear map defining $f$ on the subinterval $I_j$. The map $\tilde{f}_j$ is obtained by replacing all instances of $\lambda$ with $\lambda^{-1}$ in the formula for $f_j$. When $f$ is a zig-zag map, the constants $a_j$ are all integers, and so we have $\tilde{f}_j(y) = a_j \pm \lambda^{-1} y$. 

The Galois lift leaves invariant a natural subset $\Lambda_f \subseteq I \times \bbR$, called the \textit{limit set}. This set is a connected, finite union of Euclidean rectangles $R_j$ with horizontal and vertical sides, such that the projection onto the $x$-coordinate of the rectangle $R_j$ is

\[
\pi(R_j) = I_j \subseteq I.
\]

The action of $f_G$ on each rectangle $R_j$ is by piecewise-affine maps, and after properly identifying segments of the boundary $\partial \Lambda_f$ we obtain the action of the pseudo-Anosov $\psi_f$ acting on $S^2$. The rectangles $R_j$ form a Markov partition for $\psi_f$ in the quotient. There is a unique repelling orbit on the horizontal boundary $\partial_H \Lambda_f$, and in the quotient this orbit glues up to the point $\infty$ fixed by $\psi_f$. The minimal period of the orbit is equal to the number of prongs of the invariant foliations of $\psi_f$ at $\infty$.

The proof of Theorem \ref{mainthm:RotationNumber} proceeds as follows. We consider only the case when the number of critical points of $f$ satisfies $m \geq 4$ even. The other two cases are completely analogous.

\begin{enumerate}
\item Show that the orbit on $\partial_H \Lambda_f$ has exactly one element on each of the connected components of the \textit{lower horizontal boundary} $\partial_H^L \Lambda_f$.
\item Investigate the edge identifications of $\partial_H^L \Lambda_f$ in order to determine the counterclockwise ordering of the prongs around $\infty$.
\item Show that $\psi_f$ acts on the prongs by the permutation inverse to $\tilde{\rho}(f)$ (cf. Remark \ref{rmk:Transitive}). 
\end{enumerate}

\begin{proof}
Cut the interval $I=[0,1]$ at the set $\PC(f) \setminus \{k_1\}$ (cf. Definition \ref{defn:postcrit}) to obtain the partition

\[
I =\bigcup_{i=1}^{n-1} J_i.
\] 

\noindent Each Markov rectangle $R_j$ has a horizontal boundary

\[
\partial_H R_j = \partial_H^U R_j \sqcup \partial_H^L R_j,
\]

\noindent where $\partial_H^L R_j$ is the \textit{lower horizontal boundary of $R_j$}, i.e. the connected component of $\partial_H R_j$ that is at a lower height in $\bbR^2$. The \textit{lower horizontal boundary of $\Lambda_f$} is the union

\[
\partial_H^L \Lambda_f = \bigcup_{j=1}^n \partial_H^L R_j.
\]

\noindent From the author's work in Section 4.1 of \cite{F}, the connected components of $\partial_H^L \Lambda_f$ are precisely the sets

\[
\text{$\widetilde{J}_i = \pi^{-1}(J_i) \cap \partial_H^L \Lambda_f$, for $i=1, \ldots, n-1$.}
\]

\noindent For each $i$ the map $f_G^{-1}: \widetilde{J}_i \to f_G^{-1}  ( \widetilde{J}_i )$ is an affine contraction, with scaling factor $\lambda^{-1}$, and in fact is a homeomorphism onto its image. Moreover, we have

\begin{equation}\label{eqn:inverseperm}
f_G^{-1} (\widetilde{J}_i  ) \subseteq \widetilde{J}_{\widetilde{\rho}(f)^{-1}(i)} = \widetilde{J}_{i-(n-k)},
\end{equation}

\noindent where $\Phi(f) = \frac{n-k}{n-1}$ and $\widetilde{\rho}(f)$ is the permutation in Remark \ref{rmk:Transitive}. Here we are interpreting the index $i-(n-k)$ modulo $n-1$.

It follows now that $f_G^{-(n-1)}$ sends each $\widetilde{J}_i$ into itself, implying the existence of an attracting fixed point $p_i \in \widetilde{J}_i$. The collection $\{p_1, \ldots, p_{n-1}\}$ is necessarily the unique repelling orbit of $f_G$ on $\partial_H \Lambda_f$. We have thus located the preimages of the singular point $\infty \in S^2$ under the quotient map that identifies segments of $\partial \Lambda_f$. It remains to understand how the neighborhoods of the points $p_i$ are glued together.

We introduce a new piece of notation. Each $p_i$ is contained in the interior of $\widetilde{J}_i$, since $\pi(\partial \widetilde{J}_i) \subseteq \PC(f)$ and $\pi(p_i) \not \in \PC(f)$. Therefore, cutting $\widetilde{J}_i$ at $p_i$ produces left and right subintervals $\widetilde{J}^l_i$ and $\widetilde{J}^r_i$, respectively.

We claim that for each index $i=1, \ldots, n-2$, the edge identifications glue $\widetilde{J}_i^r$ to $\widetilde{J}_{i+1}^l$ by an orientation-reversing isometry. Once we establish this fact, the conclusion of the theorem will follow. Indeed, the standard orientation on $\bbR^2$ descends to the orientation on $S^2$, and a positive turn around $\infty \in S^2$ lifts to a sequence of counterclockwise half-turns around the points $p_i$. Each such half-turn traverses the vertical line $s_i$ emanating from the corresponding $p_i$, and $s_i$ pushes down to a prong of the contracting foliation of $\psi_f$ at $\infty$. Equation (\ref{eqn:inverseperm}) implies that

\[
f_G(p_i) = p_{i+(n-k)},
\]

\noindent but since $\widetilde{J}_i^r$ is glued to $\widetilde{J}_{i+1}^l$ we require $n-1-(n-k)=k-1$ counterclockwise half-turns to reach $s_{i+(n-k)}$ from $s_i$. It follows that $\psi_f$ rotates the prongs at $\infty$ by $k-1$ positive clicks, and hence

\[
\rot_\infty(\psi_f) = \frac{k-1}{n-1} = 1 - \frac{n-k}{n-1} = 1-\Phi(f).
\]

It remains to show that $\widetilde{J}_i^r$ glues to $\widetilde{J}_{i+1}^l$ by an orientation-reversing isometry. This is a standard argument, and is part of a broader phenomenon described, for example, in Section 3.4 of \cite{BH}, Bestvina and Handel's foundational paper on the subject of train tracks. In our case, one can see this by noting first that the segments of $\partial_H^U \Lambda_f$ adjacent to the unique periodic point in $\Lambda_f$ projecting to $k_1 \in I$ are glued in this fashion. Applying $f_G$ transports this gluing formation to all of the segments of $\partial_H^L \Lambda_f$.
\end{proof}

\section{The Farey tree}\label{sec:Farey}

In this section we introduce the Farey tree and discuss its relevant properties. For treatments from the perspective of continued fractions, we recommend \cite{K}. For further reading on the Farey tree specifically, refer to \cite{Ha}.

\subsection{The Farey sum}

All fractions $q = \frac{a}{b}$ are in lowest terms unless otherwise stated.

\begin{defn}
The \textit{Farey sum} of two fractions $\frac{a}{b}$ and $\frac{c}{d}$ is the fraction

\[
\frac{a}{b} \oplus \frac{c}{d} := \frac{a+c}{b+d},
\]
\end{defn}

\noindent where the fraction on the righthand side is not necessarily in lowest terms. Indeed, even if $\frac{a}{b}$ and $\frac{c}{d}$ are in lowest terms, their Farey sum need not be.  If, however, $|ad-bc| = 1$, then it is not hard to show that $\frac{a+c}{b+d}$ is in lowest terms. For this reason, we introduce the following non-standard terminology.

\begin{defn}
We say that two fractions $\frac{a}{b}$ and $\frac{c}{d}$ are \textit{compatible} if $|ad-bc| = 1$.  Similarly, we say that two rationals $p, q$ are \textit{compatible} if they have compatible fractional representatives. 
\end{defn}

Note that if two fractions are compatible, then they are necessarily in lowest terms. From now on, we only consider the Farey sum of compatible fractions.

The next two propositions are exercises in arithmetic.

\begin{prop}
Let $p, q$ be compatible rationals with $p<q$. Then $p < p \oplus q < q$.
\end{prop}

From now on, when we write $p \oplus q$ we will implicitly assume that $p < q$.

\begin{prop}
Let $p, q$ be compatible rationals. Then $p$ is compatible with $p \oplus q$, and $p \oplus q$ is compatible with $q$.
\end{prop}

\subsection{Constructing the Farey tree}

We are now ready to construct the Farey tree $\calF$. We do so inductively, building a new level $\calF_n$ on step at a time.\\

Set $\calF_1 = \{\frac{1}{2}\}$, the root of $\calF$.  Suppose that the first $n-1$ levels of $\calF$ have been constructed. The vertices in these levels are elements of $\bbQ \cap (0,1)$: arrange them according to the usual linear order on $\bbR$, and include $\frac{0}{1}$ and $\frac{1}{1}$ on the far left and right, respectively. To construct $\calF_n$, take the Farey sum of every element $q \in \calF_{n-1}$ with each of its two neighbors. This produces left- and right children $q_L, q_R$ of $q$, and the union of these children across all $q \in \calF_{n-1}$ is defined to be $\calF_n$. We furthermore include directed edges from each $q$ to both of its children $q_L$ and $q_R$. 

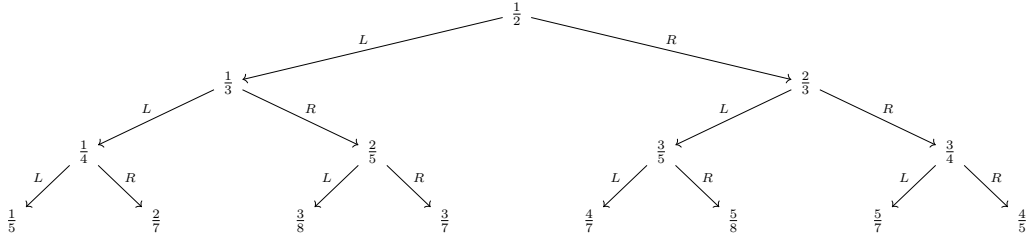
\begin{figure}
\centering
\adjustbox{scale=.7,center}{%
\begin{tikzcd}
& & & & & & & \frac{1}{2} \arrow[dllll, "L" '] \arrow[drrrr, "R"] & & & & & & & \\
& & & \frac{1}{3} \arrow[dll, "L" '] \arrow[drr, "R"] & & & & & & & & \frac{2}{3} \arrow[dll, "L" '] \arrow[drr, "R"] & & &\\
& \frac{1}{4} \arrow[dl, "L" '] \arrow[dr, "R"] & & & & \frac{2}{5} \arrow[dl, "L" '] \arrow[dr, "R"] & & & & \frac{3}{5} \arrow[dl, "L" '] \arrow[dr, "R"] & & & & \frac{3}{4}  \arrow[dl, "L" '] \arrow[dr, "R"] & \\
\frac{1}{5} & & \frac{2}{7} & & \frac{3}{8} & & \frac{3}{7} & & \frac{4}{7} & & \frac{5}{8} & & \frac{5}{7} & & \frac{4}{5}\\ 
\end{tikzcd}
}
\caption{The first four levels of the Farey tree, with edges included.}\label{fig:Farey}
\end{figure}

\begin{defn}
The \textit{Farey tree} is the directed graph $\calF$ whose vertices are $\bigcup_{n \geq 1} \calF_n$, and whose edges are precisely those described in the above procedure.
\end{defn}

\begin{rmk}
Some sources in the literature refer to $\calF$ as the \textit{Stern-Brocot tree}, while reserving the title of ``Farey tree" for a larger tree whose vertices are $\bbQ \cup \{\infty\}$. 
\end{rmk}

\subsection{Properties of the Farey tree}

A rational's position in $\calF$ is intimately related to its \textit{continued fraction expansion}. Indeed,  one can use the theory of continued fractions to construct $\calF$ and prove the following classical statement. See e.g. \cite{Ha} for more details.

\begin{prop}
The vertex set of $\calF$ is precisely $\bbQ \cap (0, 1)$. Moreover, for every vertex $r$ there is exactly one pair of compatible rationals $p, q$ such that $p \oplus q = r$. 
\end{prop}

For this reason, we slightly abuse notation and write $\calF$ for $\bbQ \cap (0,1)$.  We will also make use of the larger set

\[
\overline{\calF} = \calF \cup \left \{ \frac{0}{1}, \frac{1}{1} \right \}.
\]

\begin{defn}
Given $p, q \in \calF$, we say $p$ is an \textit{ancestor} of $q$ if the directed path from $\frac{1}{2}$ to $q$ contains $p$. We also say that $q$ is a \textit{descendant} of $p$.
\end{defn}

\begin{defn}
Let $q \in \calF_n$. The \textit{left child}  of $q$ is the unique descendant  $q_L$of $q$ in $\calF_{n+1}$ such that $q_L < q$. Similarly, the \textit{right child} is the unique descendant $q_R$ of $q$ in $\calF_{n+1}$ such that $q<q_R$.  The \textit{left-} and \textit{right parents} of $q$ are the unique elements $q^L, q^R \in \overline{\calF}$ such that $q^L \oplus q^R = q$.
\end{defn}

\begin{rmk}
The terminology of parents and children can be misleading at first. It is true that

\[
\text{$(q_L)^R = q = (q_R)^L$ for all $q \in \calF$}.
\]

\noindent On the other hand, only one of $(q^L)_R$ and $(q^R)_L$ will be equal to $q$, since each $q \neq \frac{1}{2}$ has only one incoming edge. The other of the pair will be an ancestor of $q$.
\end{rmk}

\begin{prop}\label{prop:anc_denom}
Fix $q \in \calF$. If $r \in \calF$ satisfies $q^L < r < q^R$, then $r$ is a descendant of $q$. In particular, the denominator of $r$ is at least the denominator of $q$.
\end{prop}

\begin{proof}
If $r = q$ then there is nothing to show. Therefore, assume that $r \neq q$. Suppose $q \in \calF_n$. According to the construction of $\calF$, $q$ is directly adjacent to its parents $q^L$ and $q^R$ in the set $\bigcup_{k \leq n} \calF_k$. In particular, a rational $r$ between $q$ and $q^L$ can only appear in $\calF$ after taking the mediant of $q$ and $q^L$, making $r$ a descendant of $q$. Similarly, if $r \in \calF$ is between $q$ and $q^R$, then it can only appear after taking the mediant of $q$ and $q^R$, making it once again a descendant of $q$.
\end{proof}

\begin{prop}\label{prop:parent}
Let $q \in \calF_n$ for $n \geq 2$.  Then exactly one of $q^L$ and $q^R$ is an element of $\calF_{n-1}$. Moreover,

\begin{enumerate}
\item if $q^L \in \calF_{n-1}$, then $q^R=(q^L)^R$, and
\item if $q^R \in \calF_{n-1}$, then $q^L = (q^R)^L$.
\end{enumerate}
\end{prop}

\begin{proof}
In the construction of $\calF$, the two parents of $q \in \calF_n$ are adjacent rationals in the first $n-1$ levels of $\calF$. Thus, one of them must be the parent of the other. On the other hand, there is an edge of $\calF$ pointing from some $r \in \calF_{n-1}$ to $q$: this $r$ must be a parent of $q$. Denote by $s$ the other parent of $q$. As we noted above, $s$ must be a parent of $r$: if $r=q^L$, then $r<q<s$ and it follows that $s=q^R$ and $s=r^R$. Similarly, if $r=q^R$ then $s < q < r$ and we have $s=q^L = r^L$.
\end{proof}

\begin{defn}
For $q \in \calF$, we define $\calF(q)$ to be the subtree of $\calF$ with $q$ as its root. Equivalently, $\calF(q)$ is the collection of descendants of $q$, along with $q$ itself. 
\end{defn}

Note that for each $q$ there is a natural graph isomorphism from $\calF(q)$ to $\calF$ taking $q$ to $\frac{1}{2}$ and preserving the order on $\bbR$. Thus we can speak of the levels $\calF_n(q)$ of the subtree $\calF(q)$.

\begin{prop}\label{prop:parseq}
Fix $q \in \calF$. For any $s \in \calF(q_L)$ (resp., $s \in \calF(q_R))$ there exists a sequence

\[
s = s_0, \ldots, s_k = q
\]

\noindent such that $s_i = (s_{i-1})^R$ for $i = 1, \ldots, k$ (resp., $s_i = (s_{i-1})^L$).
\end{prop}

\begin{proof}
We consider only the case $s \in \calF(q_L)$, since the other is analogous. We proceed by induction on the level $n$ of $\calF(q_L)$ containing $s$. If $n=1$, then $s = q_L$ and the sequence

\[
s_0 = q_L, s_1 = q
\]

\noindent is as desired. Now suppose that the proposition has been proven for all $s$ in the first $n-1$ levels of $\calF(q_L)$, and let $s \in \calF_n(q_L)$. By Proposition \ref{prop:parent} exactly one of $s^L$ and $s^R$ is in $\calF_{n-1}(q_L)$. If $s^R \in \calF_{n-1}(q_L)$ then we are done by the inductive hypothesis: simply prepend $s$ to the sequence of right parents connecting $s^R$ to $q$. If instead $s^L \in \calF_{n-1}(q_L)$, Proposition \ref{prop:parent} says that $s^R=(s^L)^R$. The inductive hypothesis implies that we have a sequence of right parents connecting $s^L$ to $q$, say

\[
s^L, s^R, \ldots, q.
\]

\noindent Then the desired sequence for $s$ is identical, after substituting in $s$ for $s^L$:

\[
s, s^R, \ldots, q.
\]

\noindent The proof is complete.
\end{proof}

\section{The Farey tree as a model for $\PA(m)$ and $\Pi(m)$}\label{sec:model}

The vertices of $\calF$ successively generate further vertices using the Farey sum operation. We transport this structure to both $\PA(m)$ and the family $\Pi(m)$ of pseudo-Anosovs they generate and extract dynamical consequences. The most important results from this section are Theorem \ref{thm:Farey_transform} and Proposition \ref{prop:rotprong_sum}.

\subsection{The model for $\PA(m)$}

We derive a transformational law for the principal kneading sequence $\nu(f)$ as we travel down $\calF$. Before stating this result, however, we make a few definitions.

\begin{defn}\label{def:0and1}
Fix $m \geq 2$. We will denote by $\textbf{0}_m$ the $m$-modal zig-zag map such that 

\begin{enumerate}
\item the itinerary of $x=1$ under $\textbf{0}_m$ is $(m) \cdot (m-2)^\infty$, and
\item $\textbf{0}_m$ is decreasing on $I_{m-1}$.
\end{enumerate}

\noindent We also define $\textbf{1}_m$ to be the $m$-modal zig-zag map such that

\begin{enumerate}
\item the itinerary of $x=1$ under $\textbf{1}_m$ is $(m)^\infty$, and
\item $\textbf{1}_m$ is decreasing on $I_{m-1}$.
\end{enumerate}

Set $\overline{\PA(m)} = \PA(m) \cup \{\textbf{0}_m, \textbf{1}_m\}$. We extend the map $\Phi$ to $\overline{\PA(m)}$ by declaring

\[
\Phi(\textbf{0}_m) = \frac{0}{1} \hspace{5mm} \text{and} \hspace{5mm} \Phi(\textbf{1}_m) = \frac{1}{1}.
\]
\end{defn}

\begin{defn}
Fix $m \geq 2$. We say that $f, g \in \overline{\PA(m)}$ are \textit{compatible} if $\Phi(f), \Phi(g) \in \overline{\calF}$ are compatible. In this case, we define $f \oplus g \in \PA(m)$ to be the interval map such that

\[
\Phi(f \oplus g) = \Phi(f) \oplus \Phi(g).
\]

\noindent We also borrow the language of parents and children. If $f = \Phi^{-1}(q)$, then we define the following elements of $\overline{\PA(m)}$:

\begin{itemize}
\item $f^L = \Phi^{-1}(q^L)$,
\item $f^R = \Phi^{-1}(q^R)$,
\item $f_L = \Phi^{-1}(q_L)$, and
\item $f_R = \Phi^{-1}(q_R)$.
\end{itemize}
\end{defn}

\bigskip

As Corollary \ref{cor:knead} shows, given $f \in \PA(m)$ we may write the principle kneading sequence $\nu(f)$ as the concatenation of sequences

\begin{equation}
\nu(f) = (m) \cdot \textbf{w}(f) \cdot \textbf{k}, \hspace{5mm} \text{where} \hspace{5mm} \textbf{k} = \begin{cases}
(m-1,k_1) & \text{if $m$ is even}\\
(m-1,k_1,0) & \text{if $m$ is odd.}
\end{cases}
\end{equation}

\noindent Here $\textbf{w}(f)$ is defined implicitly, and is the part of $\nu(f)$ determined by the intersection of a line segment of slope $\Phi(f)$ with horizontal and vertical integer lines in $\bbR^2$.


\begin{defn}
Fix $f \in \PA(m)$. We define the following three sequences by adusting the prefix or suffix of $\nu(f)$:

\begin{enumerate}
\item $\overline{\nu(f)} = (m) \cdot \textbf{w}(f) \cdot (m)$
\item $\widehat{\nu(f)} = (m) \cdot \textbf{w}(f) \cdot (m-2)$
\item $\underline{\nu(f)} = (m-2) \cdot \textbf{w}(f) \cdot \textbf{k}$.
\end{enumerate}
\end{defn}

\bigskip

Here are the transformation rules for $\nu(f)$.

\begin{thm}\label{thm:Farey_transform}
Let $f, g \in \PA(m)$ be compatible with $\Phi(f) < \Phi(g)$.  Then

\begin{equation}\label{eq:transform1}
\nu(f \oplus g) = \overline{\nu(f)} \cdot \underline{\nu(g)},
\end{equation}

\noindent and

\begin{equation}\label{eq:transform2}
\nu(f \oplus g) = \widehat{\nu(g)} \cdot \nu(f).
\end{equation}
\end{thm}

\begin{proof}
Write $\Phi(f)= \frac{a}{b}$ and $\Phi(g) = \frac{c}{d}$. Consider the broken line segment parameterized as

\[
\calL(t) = \begin{cases}
(t, \Phi(f) \cdot t) & \text{if $t \in [0,b]$}\\
(t, a+\Phi(g) \cdot (t-b)) & \text{if $t \in [b, b+d].$}
\end{cases}
\]

This broken line consists of two straight line segments, one of slope $\Phi(f)$ and the second of slope $\Phi(g)$.  The straight line segment $\calL_q$ shares its endpoints with $\calL$, but lies above it in the plane. Consider the isotopy $\calH$ from $\calL$ to $\calL_q$ defined by dragging the non-smooth point of $\calL$ vertically towards $\calL_q$, while keeping this point connected by line segments to the endpoints $(0,0)$ and $(b+d, a+c)$. Pick a parameterization of this isotopy by $s \in [0,1]$ such that $\calH_0(\calL) = \calL$ and $\calH_1(\calL) = \calL_q$. 

We can define the \textit{kneading sequence} $\nu(s)$ of $\calH_s(\calL)$ using the intersections with horizontal and vertical integer lines, going from left to right. Explicitly, we declare the following algorithm:

\begin{enumerate}
\item Initialize $\nu(s) = (m)$.

\item For each $1 \leq i \leq b+d-2$:

\begin{enumerate}
\item if the intersection of $\calH_s(\calL)$ with $x=i+1$ is an integer point, append $\textbf{k}$ to $\nu(s)$;
\item else if the intersection of $\calH_s(\calL)$ with $x=i$ is an integer point, append $m$ to $\nu(s)$;
\item else if there is a horizontal intersection strictly between $x=i$ and $x=i+1$, append $m$ to $\nu(s)$;
\item else, append $m$ to $\nu(s)$;
\end{enumerate}
\item Append $\textbf{k}$ to $\nu(s)$.
\end{enumerate}

According to this algorithm, the initial and terminal kneading sequences of the isotopy satisfy

\begin{align*}
\nu(0) & = \nu(f) \cdot \nu(g),\\
\nu(1) & = \nu(f \oplus g).
\end{align*}

Moreover, the value of $\nu(s)$ only changes when one of the line segments shifts across an integer point $(x,y) \in \bbZ^2$. Indeed, suppose the left line segment of $\calH_{s_0}(\calL)$ passes through the integer point $(x,y) \neq (0,0)$, for some $s_0 > 0$.  Then this left line segment has slope $y/x$, which is equal to some rational $r$ with denominator dividing $x$ and therefore at most $b$, since $x \leq b$. But this rational $r$ is in the interval $[\Phi(f), \Phi(g)]$, so Proposition \ref{prop:anc_denom} implies that $x$ is at least as large as the denominator of $\Phi(f) \oplus \Phi(g)$, which is $b+d$. Since $d \geq 1$, we have a contradiction: the left line segment of $\calH_{s_0}(\calL)$ cannot pass through a second integer point besides $(0,0)$. An identical argument proves that the right line segment of $\calH_{s_0}(\calL)$ cannot pass through a second integer point besides $(a+c, b+d)$.

Therefore $\nu(s) = \nu(f \oplus g)$ for all $s > 0$. The initial perturbation of $\calH_0(\calL)$ replaces the suffix $\textbf{k}$ of $f$ with $(m)$ and the prefix $m$ of $\nu(g)$ with $m-2$, verifying (\ref{eq:transform1}). See Figure \ref{fig:Fareycut2}.

To prove (\ref{eq:transform2}) we proceed by an identical argument, except that we define $\calL$ to be the broken line whose first segment has slope $\Phi(g)$ and whose second segment has slope $\Phi(f)$. In this case, the initial isotopy replaces the suffix $\textbf{k}$ of $\nu(g)$ with $(m-2)$, while it does not alter the prefix $(m)$ of $\nu(f)$. See Figure \ref{fig:Fareycut1}.
\end{proof}

\begin{rmk}\label{rmk:edges}
The formulas are slightly different if $f=\textbf{0}_m$ or $g = \textbf{1}_m$. In the first case, $g= \Phi^{-1}(1/n)$ for some $n$, and so $\textbf{w}(g) = (m-2)^{n-2}$. It is not hard to see then that

\[
\textbf{w}(\textbf{0}_m \oplus g) = (m-2)^{n-1}.
\]

\noindent In the second case, $f=\Phi^{-1}(1-1/n)$ and $\textbf{w}(f) = (m)^{n-2}$. Therefore,

\[
\textbf{w}(f \oplus \textbf{1}_m) = (m-1)^{n-1}.
\]
\end{rmk}

 
From now on, when we write $\nu(q)$ for $q \in \calF$ we will mean $\nu \left (\Phi^{-1}(q) \right )$. Theorem \ref{thm:Farey_transform} allows us to compute $\nu(q)$ as we travel along $\calF$ from $\frac{1}{2}$ to $q$.

\begin{ex}
We compute $\nu(7/12)$ step by step. The path in $\calF$ from $1/2$ to $7/12$ is pictured in Figure \ref{fig:Fareyex}. Dotted arrows depict a parent-child relationship that is not an edge of $\calF$.

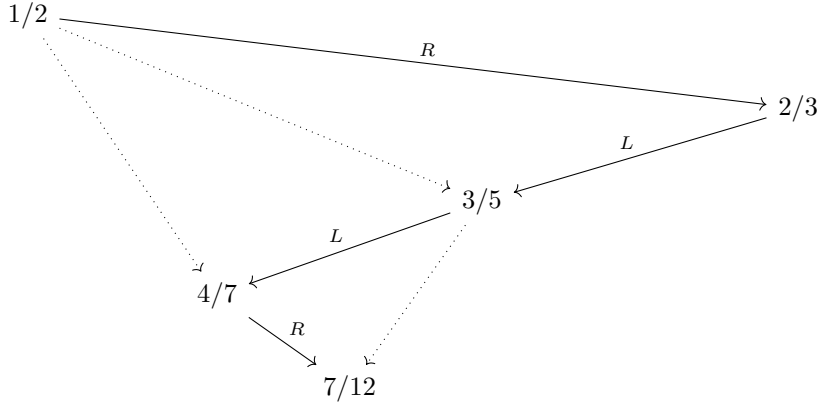
\begin{figure}[!ht]
\centering
\begin{tikzcd}
	1/2 
	\arrow[ddrrrr, dotted]
	\arrow[dddrr, dotted]
	\arrow[drrrrrrrr, "R"]\\
		& & & & & & & & 2/3 \arrow[dllll, "L" ']\\
		& & & & 3/5 \arrow[dll, "L" ']
				\arrow[ddl, dotted]\\
		& & 4/7 \arrow[dr, "R"]\\
		& & & 7/12
\end{tikzcd}
\caption{The unique path in the Farey tree from the root $1/2$ to $7/12$. Dotted arrows depict a parent-child relationship that is not an edge of $\calF$.}\label{fig:Fareyex}
\end{figure}

We begin with $\nu (1/2) = (m) \cdot \textbf{k}$.  Remark \ref{rmk:edges} tells us that 

\[
\nu \left ( \frac{2}{3} \right ) = (m) \cdot (m) \cdot \textbf{k}.
\]

\noindent The left child of $2/3$ is $3/5$, which has left parent $1/2$, hence by (\ref{eq:transform2})

\[
\nu \left ( \frac{3}{5} \right ) = \widehat{\nu \left ( \frac{2}{3} \right )} \cdot \nu \left ( \frac{1}{2} \right) = (m) \cdot (m,m-2,m) \cdot \textbf{k}.
\]

\noindent Similarly, the left child of $3/5$ is $4/7$, whose left parent is again $1/2$, giving

\[
\nu \left ( \frac{4}{7} \right ) = \widehat{\nu \left ( \frac{3}{5} \right )} \cdot \nu \left ( \frac{1}{2} \right) = (m) \cdot (m,m-2,m,m-2,m) \cdot \textbf{k}.
\]

\noindent See Figure \ref{fig:Fareycut1}. Finally we arrive at $7/12$, whose left- and right parents are $4/7$ and $3/5$, respectively. Applying (\ref{eq:transform1}) gives

\begin{align*}
\nu \left ( \frac{7}{12} \right ) = \overline{\nu \left ( \frac{4}{7} \right )} \cdot \underline{\nu \left ( \frac{3}{5} \right )} & = (m) \cdot (m,m-2,m,m-2,m,m) \cdot (m-2,m,m-2,m) \cdot \textbf{k}\\
& = (m) \cdot (m,m-2,m,m-2,m,m,m-2,m,m-2,m) \cdot \textbf{k}.
\end{align*}

\noindent See Figure \ref{fig:Fareycut2}.
\end{ex}

\begin{figure}
\begin{subfigure}{.9\textwidth}
\centering
\includegraphics[scale=.3]{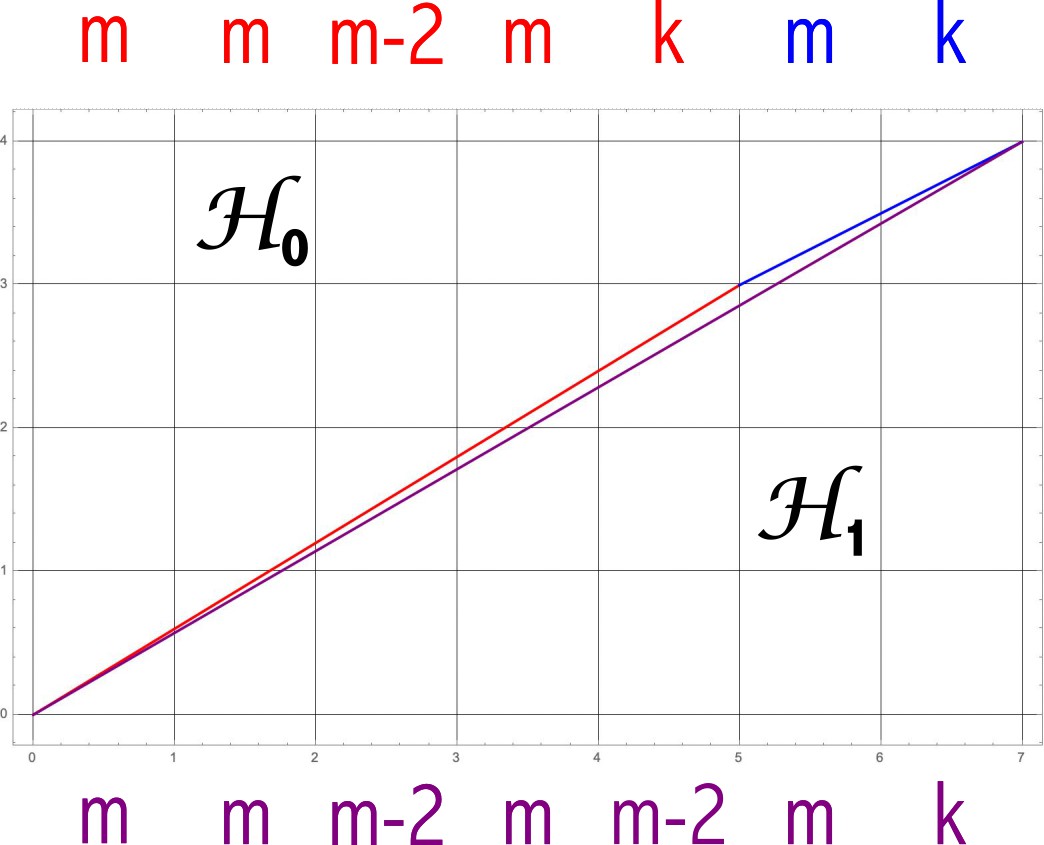}
\caption{Computing $\nu(4/7)$ from $\nu(3/5)$ and $\nu(1/2)$.}\label{fig:Fareycut1}
\end{subfigure}
\par \bigskip
\begin{subfigure}{.9\textwidth}
\centering
\includegraphics[scale=.35]{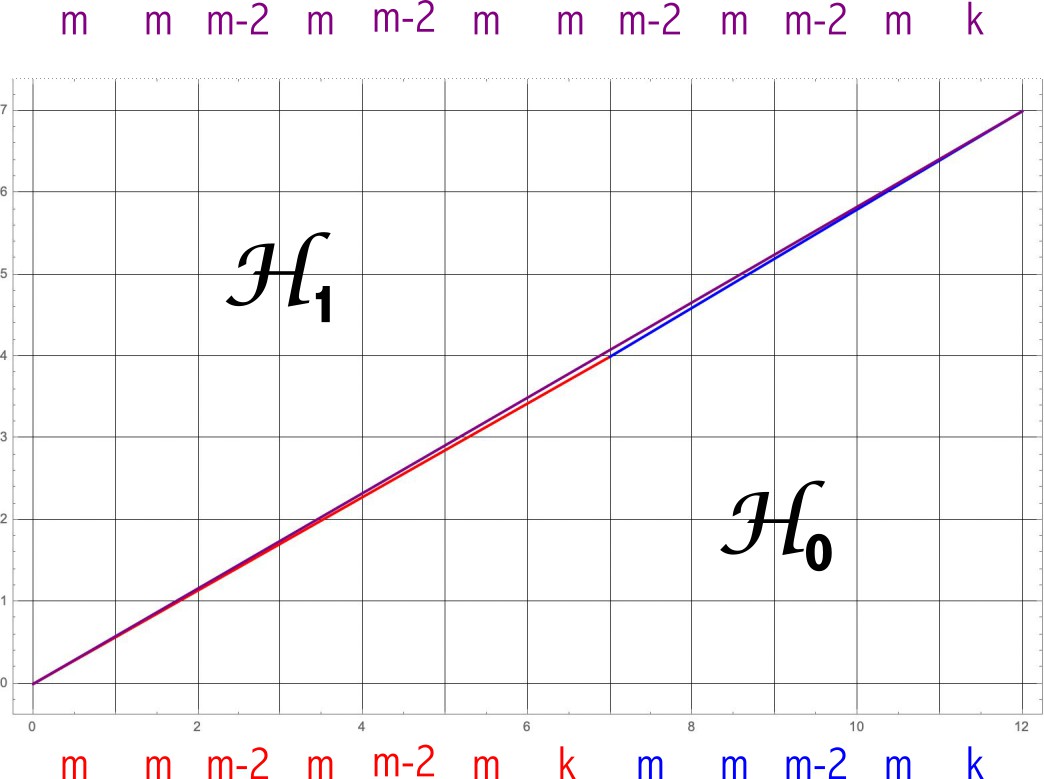}
\caption{Computing $\nu(7/12)$ from $\nu(4/7)$ and $\nu(3/5)$.}\label{fig:Fareycut2}
\end{subfigure}
\caption{The last two steps in computing $\nu(7/12)$.}\label{fig:Fareycut}
\end{figure}

\begin{rmk}\label{rmk:perturb}
Before moving on to $\Pi(m)$, we pause to reflect on the ramifications of Theorem \ref{thm:Farey_transform} and Remark \ref{rmk:edges}. We may interpret these as saying that making a turn in $\calF$ at $\Phi(f)$ amounts to \textit{perturbing} the point $f^{b-1}(1) \in I_{m-1}$. Indeed, suppose that $\Phi(f) = \frac{a}{b}$. Then for any left child $g$ and any right child $h$ we have

\begin{equation}\label{eq:transform3}
  \begin{alignedat}{3}
    &\Pre_b(\It_g(1)) &&= \widehat{\nu(f)} && = (m) \cdot \textbf{w}(f) \cdot (m-2)\\
    & \Pre_b(\It_f(1)) && = \Pre_b(\nu(f)) && = (m) \cdot \textbf{w}(f) \cdot (m-1)\\
    &\Pre_b(\It_h(1)) &&= \overline{\nu(f)} && = (m) \cdot \textbf{w}(f) \cdot (m)
  \end{alignedat}
\end{equation}

\noindent Descending to the left (resp. right) in $\calF$ perturbs $f^{b-1}(1)$ to the left (resp., right). As we will see in Section \ref{sec:mono}, it is this behavior that causes the stretch factor $\lambda(f)$ to change monotonically with $\Phi(f)$.
\end{rmk}

\bigskip

Finally, it is instructive to graph the association of a rational number $q$ with the stretch factor of the corresponding $f = \Phi^{-1}(q) \in \PA(m)$, for a fixed $m$.  Figure \ref{fig:entropy} does this for the case $m=2$. As Theorem \ref{mainthm:mono} claims, the stretch factor increases monotonically with the rotation number. Additionally, we see that each stretch factor has a gap on either side separating it from other stretch factors. Indeed, if $\{q_k\}_k$ is a sequence of rationals limiting to $q \in \calF$, with $q_k \neq q$ for all $k$, then the stretch factor of $\Phi^{-1}(q_k)$ is bounded away from that of $q$. 

One should think of this phenomenon as a manifestation of the Farey tree structure on $\PA(m)$.  For example, the left children of any $q \in \calF$ limit to $q$ from below: that is, the children of $q$ are arbitrarily close to $q$ in the topology of $\bbR$. In the Farey tree, however, these children remain far away from $q$. The deeper down $\calF$ $q$ appears, the smaller the gap between $q$ and its children becomes, but the gap is always there. Experimental evidence shows that these gaps are not symmetric.  

A similar plot appears as Figure 12 in Hall's thesis \cite{H}, which among other things analyzes the case $m=1$. In that figure the entropy is monotonically \textit{decreasing}, but this difference is merely an artifact of the definition of $\Phi(f)$. Indeed, if we define $\Phi'(f) = 1-\Phi(f)$, then $\Phi'$ is monotonically decreasing and is precisely equal to $\rot_\infty(\psi_f)$, by Theorem \ref{mainthm:RotationNumber}.

\begin{figure}
\centering
\includegraphics[scale=0.4]{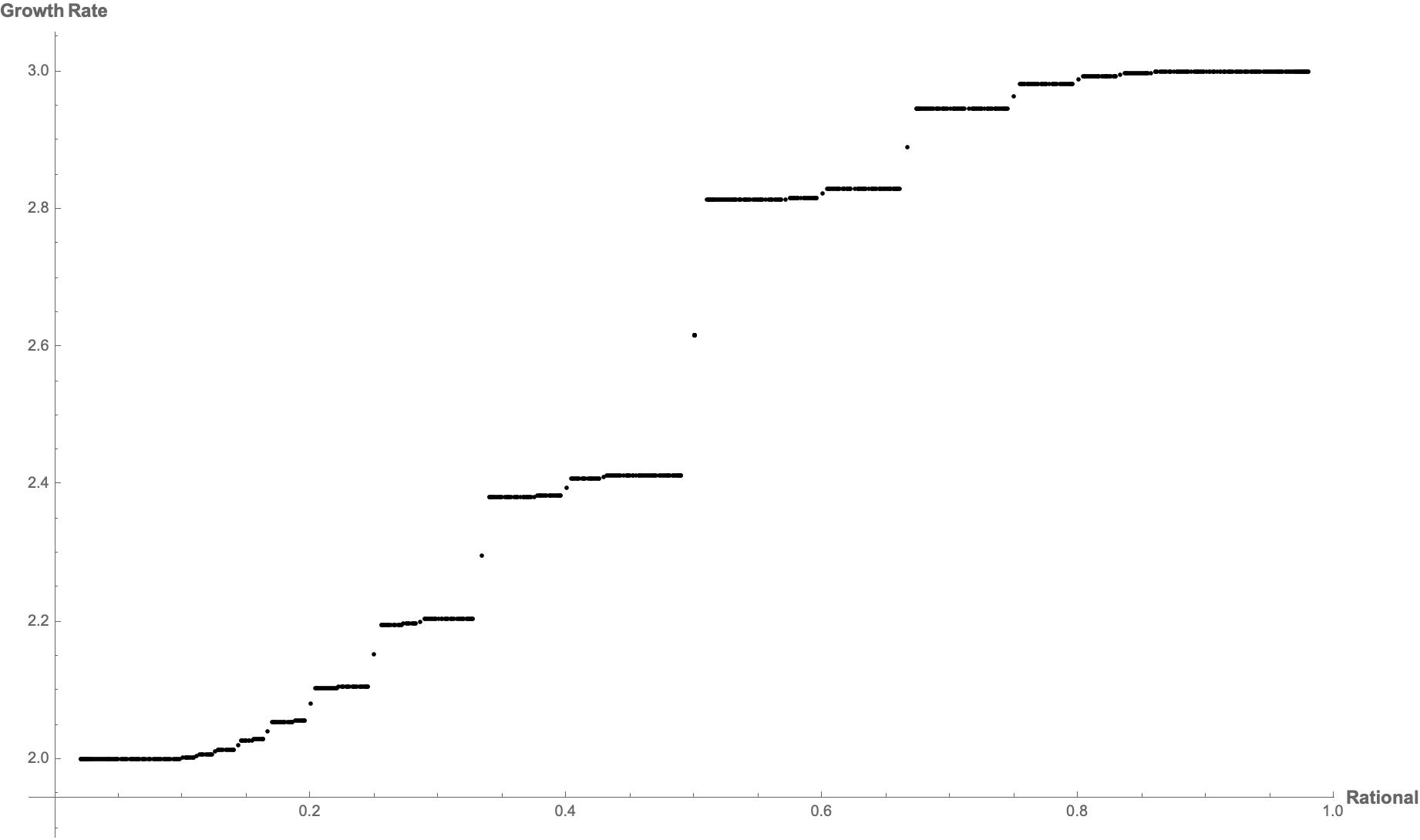}
\caption{We graph the stretch factor of elements of $\PA(2)$. Given $f \in \PA(2)$,  we plot the point $(\Phi(f), \lambda(f)$. Each stretch factor is contained in a gap, and rational}\label{fig:entropy}
\end{figure}

\subsection{The model for $\Pi(m)$}

Theorem \ref{thm:Farey_transform} and Remark \ref{rmk:edges} have implications for how the Markov partition and transition matrix of $f = \Phi^{-1}(q) \in \PA(m)$ transform as $q \in \calF$ varies. These transformation laws are reincarnated in the Markov partition and transition matrix of $\psi(f) \in \Pi(m)$. For brevity, however, we spend this subsection describing a dynamico-topological transformation law that is more meaningful from the perspective of $\Pi(m)$.

\begin{defn}
If $f, g \in \PA(m)$ are compatible, then we say that the pseudo-Anosovs $\psi_f, \psi_g \in \Pi(m)$ are \textit{compatible}, as well. In this case, define $\psi_f \oplus \psi_g$ to be the element $\Psi(f \oplus g) \in \Pi(m)$.
\end{defn}

\begin{defn}
Given $f \in \PA(m)$, denote by $n_\infty(\psi_f)$ the number of prongs of the invariant foliations of $\psi_f$ at the singularity $\infty$.
\end{defn}

Note that if $\Phi(f) = \frac{a}{b}$ then $n_\infty(\psi_f) = b$.

\begin{prop}\label{prop:rotprong_sum}
Fix $m \geq 2$, and let $\phi, \psi \in \Pi(m)$ be compatible. Then

\begin{equation}\label{eq:rot_sum}
\rot_\infty(\phi \oplus \psi) = \rot_\infty(\phi) \oplus \rot_\infty(\psi).
\end{equation}

\noindent Moreover,

\begin{equation}\label{eq:prong_sum}
n_\infty(\phi \oplus \psi) = n_\infty(\phi) + n_\infty(\psi).
\end{equation}
\end{prop}

\begin{proof}
These relations follow directly from our definitions and Theorem \ref{mainthm:RotationNumber}. Let $f, g \in \PA(m)$ be the compatible maps generating $\phi$ and $\psi$, respectively. Write $\Phi(f) = \frac{a}{b}$ and $\Phi(g) = \frac{c}{d}$. Then

\begin{align*}
\rot_\infty(\phi \oplus \psi) & = 1 - \left ( \frac{a}{b} \oplus \frac{c}{d} \right)\\
& = \frac{b+d-(a+c)}{b+d} \\
& = \frac{(b-a) + (d-c)}{b+d}\\
& = \frac{b-a}{b} \oplus \frac{d-c}{d}\\
& = \left (1-\frac{a}{b} \right ) \oplus \left ( 1 - \frac{c}{d} \right )\\
& = \rot_\infty(\phi) \oplus \rot_\infty(\psi).
\end{align*}

\noindent This proves Equation (\ref{eq:rot_sum}). Since for any $h \in \PA(m)$ the fraction $\Phi(h)$ is always in lowest terms, and since the number of prongs of $\psi_h$ at infinity is equal to the denominator of $\Phi(h)$, Equation (\ref{eq:prong_sum}) now follows.
\end{proof}

\section{Monotonicity of entropy}\label{sec:mono}

Recall that the topological entropy of a $\lambda$-zig-zag map $f$ is $\h(f) = \log(\lambda)$. In the study of interval dynamics, many one-parameter families of maps have been shown to exhibit \textit{monotonicity of entropy}: that is, entropy monotonically increases or decreases with the parameter. In this section we prove Theorem \ref{mainthm:mono}, which can be viewed as another example of this phenomenon. Here, the parameter is the anti-rotation number $\Phi: \PA(m) \to \calF$.

\begin{mainthm}
Let $f, g \in \PA(m)$. Then

\[
\Phi(f) < \Phi(g) \iff \lambda(\psi_f) < \lambda(\psi_g).
\]
\end{mainthm}

\begin{proof}
The result follows from the chain of equivalences:

\begin{align*}
\Phi(f) < \Phi(g) & \iff \It_f(1) <_E \It_g(1) \tag{\ref{prop:cutkneadmono}}\\
& \iff \calK_f \ll \calK_g \tag{\ref{prop:itmono}}\\
& \iff \lambda(f) < \lambda(g) \tag{\ref{cor:itmono}}\\
& \iff \lambda(\psi_f) < \lambda(\psi_g).
\end{align*}

\noindent The last equivalence is because $\lambda(f) = \lambda(\psi_f)$.
\end{proof}

In Section \ref{subsec:lexico} we introduce the \textit{twisted lexicographic order} $\leq_E$, and in Section \ref{subsec:proofmono} we prove Proposition \ref{prop:cutkneadmono} and Corollary \ref{cor:itmono}.

\subsection{Ordering kneading sequences}\label{subsec:lexico}

In this section we recall the twisted lexicographic ordering of a continuous piecewise-monotone map $f: I \to I$ with finitely many critical points $k_1, \ldots, k_m$. We call such a map \textit{$m$-modal}. All definitions rely on the data of a given such $f$, which we therefore often suppress from the notation. We are interested in the case that $f$ is a zig-zag map.\\

Refer to Section \ref{subsection:Knead} for the definitions of address, itinerary, kneading data, and principal kneading sequence.

\begin{defn}
The \textit{sequence space} of $f$ is the set

\[
\Sigma_f = \{ \It_f(x) : x \in I\}.
\]

\noindent The \textit{shift} on $\Sigma_f$ is the unique map $\sigma: \Sigma_f \to \Sigma_f$ such that the following diagram commutes:

\begin{center}
\begin{tikzcd}
I \arrow[r, "f"] \arrow[d, "\It_f" '] & \arrow[d, "\It_f"] I \\
\Sigma_f \arrow[r, "\sigma"] & \Sigma_f
\end{tikzcd}
\end{center}

\noindent More concretely, $\sigma$ shifts all entries of $\It_f(x)$ one space to the left, deleting the first entry.

\end{defn}

\begin{defn}
The \textit{sign function} of $f$ is the function $E: \calA \to \{-1, 0, +1\}$ such that

\[
E(j) = \begin{cases}
+1 & \text{if $f$ is increasing on $I_j$}\\
-1 & \text{if $f$ is decreasing on $I_j$}\\
0 & \text{if $j=k_i$ is a critical point.}
\end{cases}
\]

\noindent Note that since $f$ is continuous, $E$ is determined by $E(0)$. In the following discussion we work with a fixed $f$ and the sign function $E$ it defines.
\end{defn}

\begin{ex}\label{ex:sign}
Suppose $f \in \PA(m)$ with sign vector $E$. Then for $0 \leq j \leq m$, 

\[
E(j) = \begin{cases}
(-1)^j & \text{if $m$ is even}\\
(-1)^{j+1} & \text{if $m$ is odd.}
\end{cases}
\]

\noindent In particular, $E(m) = +1$ always.
\end{ex}
 
 \begin{defn}
 The \textit{cumulative sign vector} of a (possibly infinite) word $A$ in the alphabet $\calA_m$ is the sequence $(s_i(A))_i$ such that 
 
 \[
 s_i(A)=\begin{cases}
 1 & \text{if $i=0$}\\
E(A_{i-1}) \cdot s_{i-1}(A) & \text{if $i \geq 1$.}
 \end{cases}
 \]
 \end{defn}
 
\begin{rmk}\label{rmk:possign}
Consider $f \in \PA(m)$ and write $\Phi(f) = \frac{a}{b}$. From Corollary \ref{cor:knead} we know that the entries of the principal kneading sequence of $f$ satisfy

\[
\text{$\nu_i(f) \in \{m-2, m\}$ for all $0 \leq i \leq b-2$.}
\]

\noindent As Example \ref{ex:sign} shows, $E(m) = E(m-2) = +1$ for all values of $m$. Therefore, the cumulative signs of $\nu(f)$ satisfy

\begin{equation}\label{eq:possign}
\text{$s_i(\nu(f)) = +1$ for all $0 \leq i \leq b-1$.}
\end{equation}
\end{rmk}

We now define a partial order on $\Sigma_f$.

\begin{defn}\label{defn:lexico}
Let $f$ be an $m$-modal map with sign vector $E$. We define the \textit{twisted lexicographic order} $\leq_E$ on $\Sigma_f$ as follows. First, we declare

\[
0 <_E k_1 <_E 1 <_E k_2 <_E 2 <_E \cdots <_E k_m <_E m.
\]

\noindent For $A, B \in \Sigma_f$ we now define $A \leq_E B$ if either $A=B$ or else, for the maximal index $l \geq 0$ such that $\Pre_l(A)=\Pre_l(B)$, either

\[
\begin{cases}
A_{l+1} <_E B_{l+1} & \text{if $s_l(A) = +1$, or}\\
A_{l+1} >_E B_{l+1} & \text{if $s_l(A) = -1$.}
\end{cases}
\]
\end{defn}

While we technically defined the ordering $\leq_E$ on the shift space for a single map $f$, observe that this definition only relies on the input data of the sign function $E$. Therefore, it makes sense to compare sequences from the shift spaces of two maps $f$ and $g$ with the same sign function $E$. With this in mind, we now define a partial order on the space of kneading data of maps with the same sign function.

Recall that if $f$ is an $m$-modal map, then the \textit{kneading data} of $f$ is the ordered tuple of critical itineraries of $f$:

\[
\calK_f = (\calK_1(f), \ldots, \calK_m(f)).
\]

\begin{defn}
Let $f, g$ be $m$-modal maps with the same sign function $E$. We say that $\calK_f \ll \calK_g$ if for all $j=1, \ldots, m$ we have

\begin{equation}\label{eqn:knead}
\begin{cases}
\calK_j(f) \leq_E \calK_j(g) & \text{if $E(j) = -1$, and}\\
\calK_j(f) \geq_E \calK_j(g) & \text{if $E(j) = +1$.}
\end{cases}
\end{equation}
\end{defn}

The following proposition shows that this ordering on kneading data has implications for the topological entropies of maps with the same sign function.

\begin{prop}[Corollary 4.5 in \cite{MTr}]\label{prop:kneadmono}
If $\calK_f \ll \calK_g$, then $\h(f) \leq \h(g)$.
\end{prop}

\subsection{Completing the proof of Theorem \ref{mainthm:mono}}\label{subsec:proofmono} In this section we synthesize our work from Sections \ref{sec:Farey} and \ref{sec:model} to prove that the stretch factor $\lambda(f)$ grows monotonically in the anti-rotation number $\Phi(f)$.

\begin{prop}\label{prop:Farey_itin}
For all $f \in \PA(m)$ we have

\[
\It_{f^L}(1) <_E \It_f(1) <_E \It_{f^R}(1).
\]
\end{prop}

\begin{proof}
In the case that both $f^L$ and $f^R$ are in $\calF$, these inequalities follow directly from Equations (\ref{eq:transform3}) and (\ref{eq:possign}). For the special case where $f^L = \textbf{0}_m$ or $f^R = \textbf{1}_m$, we instead appeal to the formulas in Remark \ref{rmk:edges}.
\end{proof}

We are now ready to compare the kneading data of elements of $\PA(m)$.

\begin{prop}\label{prop:cutkneadmono}
Fix $m \geq 2$, and let $f$ and $g$ be distinct elements of $\PA(m)$. Then

\[
\Phi(f) < \Phi(g) \iff \It_f(1) <_E \It_g(1).
\]
\end{prop}

\begin{proof}
Let $r$ be the latest common ancestor of $\Phi(f)$ and $\Phi(g)$, and set $h = \Phi^{-1}(r) \in \PA(m)$. 

Suppose that $\Phi(f) < \Phi(g)$. In the general case that $r$ is distinct from $\Phi(f)$ and $\Phi(g)$, we have $\Phi(f) \in \calF(r_L)$, implying by Proposition \ref{prop:parseq} that there exists a sequence of right parents

\[
s_0 = \Phi(f), s_1, \ldots, s_k = r
\]

\noindent connecting $\Phi(f)$ to $r$. Setting $f_i = \Phi^{-1}(s_i)$, we have $f_i^R = f_{i+1}$ for $i = 0, \ldots, k-1$. Proposition \ref{prop:Farey_itin} now implies that

\[
\It_f(f) <_E \It_{f_1}(1) <_E \cdots <_E \It_h(1).
\]

\noindent A similar argument involving left parents shows that $\It_h(1) <_E \It_g(1)$.

In the special case that $r=\Phi(f)$ or $\Phi(g)$, we have $f=h$ or $g=h$, respectively. From here we apply the same argument to obtain the same conclusion.

Suppose instead that $\Phi(f) \geq \Phi(g)$. Then in fact $\Phi(f) > \Phi(g)$, and repeating our above argument shows that $\It_f(1) >_E \It_g(1)$.
\end{proof}

\begin{prop}\label{prop:itmono}
For any $m \geq 2$, let $f$ and $g$ be distinct elements of $\PA(m)$. Then

\[
\calK_f \ll \calK_g \iff \It_f(1) <_E \It_g(1).
\]
\end{prop}

\begin{proof}
Suppose first that $m$ is even. Then $E(j) = (-1)^j$, so the conditions in (\ref{eqn:knead}) become

\[
\calK_f \ll \calK_g \iff \begin{cases}
\It_f(1) \leq_E \It_g(1) & \text{for $j$ odd, and}\\
0^\infty \geq_E 0^\infty & \text{for $j$ even.}
\end{cases}
\]

\noindent Since $f$ and $g$ are distinct elements of $\PA(m)$ they cannot have the same itinerary, and so the conclusion holds in this case.

Suppose instead that $m$ is odd. Then $E(j) = (-1)^{j+1}$, so the conditions in (\ref{eqn:knead}) are now

\[
\calK_f \ll \calK_g \iff \begin{cases}
\It_f(1) \leq_E \It_g(1) & \text{for $j$ odd, and}\\
0 \cdot \It_f(1) \geq_E 0 \cdot \It_g(1) & \text{for $j$ even.}
\end{cases}
\]

\noindent Since $E(0)=-1$, the conclusion again holds. 
\end{proof}

\begin{cor}\label{cor:itmono}
For any $m \geq 2$, let $f$ and $g$ be distinct elements of $\PA(m)$. Then

\[
\lambda(f) < \lambda(g) \iff \It_f(1) <_E \It_g(1).
\]
\end{cor}

\begin{proof}
If $\It_f(1) <_E \It_g(1)$, then $\calK_f \ll \calK_g$ by Proposition \ref{prop:itmono}, hence by Proposition \ref{prop:kneadmono} we have

\[
\log{\lambda(f)} = \h(f) \leq h(g) = \log{\lambda(g)}.
\]

\noindent This inequality must in fact be strict, since elements of $\PA(m)$ have distinct slopes and hence distinct entropies. If $\It_f(1) \geq_E \It_g(1)$, then in fact $\It_f(1) > \It_g(1)$, and the same arguments imply $\lambda(f) > \lambda(g)$.
\end{proof}

\section{The digit polynomial as other invariants}\label{sec:poly_invts}

We conclude by showing that the digit polynomial of a map $f \in \PA(m)$ recovers several well-known invariants. These include:

\begin{enumerate}
\item the \textit{strong Markov polynomial} $\chi_M(f; t)$ of $f$ (Section \ref{subsec:Markov}), 
\item the \textit{Artin-Mazur zeta function} $\zeta_f(t)$ of $f$ (Section \ref{subsec:zeta}), 
\item the \textit{homology and symplectic polynomials} of $\psi_f$ (Section \ref{subsec:BBK}), and
\item the \textit{symplectic Burau polynomial} $\chi_\beta(-1, t)$ of any braid representative $\beta$ for $\psi_f$ (Section \ref{subsec:Burau}).
\end{enumerate}

Importantly, all of these polynomials are essentially the characteristic polynomial of the \textit{strong Markov matrix} for $f$, which may be understood as the transition matrix for $\psi_f$ acting on a particular invariant train track.

\begin{mainthm}
Let $f \in \PA(m)$ and set $n = |\PC(f)| = 1+\deg(D_f)$. Set $\lambda = \lambda(f) = \lambda(\psi_f)$. 

\begin{enumerate}
\item \label{poly_invts:zeta} The digit polynomial determines the Artin-Mazur zeta function of $f$: 

\[
\zeta_f(t) = \frac{1}{\calR(D_f(t))} = \frac{1}{\calR(\det(tI-W_f))}.
\]

\item \label{poly_invts:chi} The digit polynomial is equal to the strong Markov polynomial of $f$:

\[
D_f(t) = \chi_M(f; t).
\]

\item \label{poly_invts:hom_1} The digit polynomial determines the homology, symplectic, and puncture polynomials of $\psi_f$:

\[
D_f(t) = h(\psi_f; t) = \begin{cases}
s(\psi_f; t) & \text{if $n$ is odd}\\
s(\psi_f; t)(t+1) & \text{if $n$ is even.}
\end{cases}
\]

\item \label{poly_invts:hom_2} Let $\bbS$ be the surface obtained from the orientation cover of $\psi_f$ by filling in the lifts of the punctured $1$-prong singularities of $\psi_f$. Denote by $\chi_+(t)$ (resp., $\chi_-(t)$) the characteristic polynomial of $\psi_+$ (resp., $\psi_-$) acting on $H_1(\bbS; \bbZ)$. Then

\[
D_f(t) = \chi_+(t) \hspace{5mm} \text{and} \hspace{5mm} D_f(-t) = \chi_-(t).
\]


\item \label{poly_invts:Burau} Let $\beta_f$ be any $n$-braid representative of $\psi_f$ obtained by ripping open $\infty \in S_{0,n+1}$ to a boundary circle. Let $\bbB(\beta_f, z)$ denote the reduced Burau matrix for $\beta_f$, and set $\chi(\beta_f; t) = \det(tI-\bbB(\beta_f, -1))$. Then

\[
\chi(\beta_f; t) = \begin{cases}
D_f(t) & \text{if $\chi(\beta_f; t)$ has $\lambda$ as a root}\\
D_f(-t) & \text{if $\chi(\beta_f; t)$ has $-\lambda$ as a root.}
\end{cases}
\]

\noindent Moreover, composing $\beta_f$ with the full twist $\Delta_n^2$ negates the sign of the variable $t$.
\end{enumerate}
\end{mainthm}




\subsection{Markov partitions}\label{subsec:Markov}

From a postcritically finite interval map $f$ one obtains a natural partition by cutting $I$ at the \textit{weak postcritical set} of $f$:

\[
\WPC(f) : = \PC(f) \cup \{\text{the critical points of f}\}.
\]

\noindent We refer to the resulting partition of $I$ as the \textit{weak Markov partition} of $f$. If $|\WPC(f)|=r$ then from the resulting subintervals $I_1, \ldots, I_{r-1}$ we form the \textit{weak Markov matrix} $W_f$ of $f$ as follows:

\[
(W_f)_{i,j} = \begin{cases}
1 & \text{if $\overline{f(I_j)} \supseteq \overline{I_i}$}\\
0 & \text{otherwise.}
\end{cases}
\]

It is well known that the spectral radius of $W_f$ is $\lambda = e^{\h(f)}$. The Perron-Frobenius theorem says that $\lambda$ is also an eigenvalue of $W_f$, and therefore a root of the \textit{weak Markov polynomial} of $f$:

\[
\chi_W(f; t) = \det(tI - W_f).
\]

While the weak Markov partition has several nice properties, it is larger than necessary if some critical point of $f$ is not in $\PC(f)$. If we only cut $I$ at the elements of $\PC(f)$ then we obtain the \textit{strong Markov partition} of $f$. The corresponding \textit{strong Markov matrix} $M_f$ of $f$ is defined analogously to $W_f$:

\[
(M_f)_{i,j} = \text{the number of times $\overline{f(J_j)}$ traverses $\overline{J_i}$.}
\]

The \textit{strong Markov polynomial} of $f$ is $\chi_M(f; t) = \det(tI - M_f)$. An exercise in linear algebra demonstrates the following relationship.

\begin{prop}\label{prop:Markov_poly}
Set $r = |\WPC(f)|$ and $n = |\PC(f)|$. Then $\chi_W(f; t) = t^{r-n} \chi_M(f; t)$. In particular, $e^{\h(f)}$ is a root of $\chi_M(f; t)$.
\end{prop}

In \cite{F} the author shows that $M_f$ is the more natural matrix to consider for an interval map of pseudo-Anosov type. Indeed, the points of $\PC(f)$ correspond to the $1$-prong singularities of $\psi_f$, and the subintervals between these points become the expanding edges of a train track for $\psi_f$. In particular $M_f$ is also a Markov transition matrix for $\psi_f$. Thus, the dynamics of $f$ determine the dynamics of $\psi_f$ as a Bernoulli system.

\subsection{The reverse of a polynomial}

The \textit{reverse} of a polynomial $p(t) \in \bbC[t]$ is

\[
\calR(p(t)) = t^{\deg(p)} p(t^{-1}).
\]

\noindent The polynomial $p(t)$ is \textit{reciprocal} or \textit{symmetric} if $\calR(p(t)) = p(t)$. 

\begin{rmk}
Note that $D_f(t)$ is reciprocal, by Lemma \ref{fundlem:Digit}.
\end{rmk}

We will need the following fact.

\begin{prop}\label{prop:rev_poly}
For any $f, g \in \bbC[t]$ we have

\[
\calR(f(t)) = \calR(g(t)) \iff f(t) = t^{\deg(f)-\deg(g)} g(t).
\]
\end{prop}
\subsection{Zeta functions of zig-zag maps}\label{subsec:zeta}

In this section we prove statements (\ref{poly_invts:zeta}) and (\ref{poly_invts:chi}) of Theorem \ref{mainthm:poly_invts}.

Suppose that $f: X \to X$ is a dynamical system such that the sets

\[
\Fix(f^i) = \{x \in X: f^i(x) = x\}
\]

\noindent are finite for each $i \geq 1$. Then the \textit{Artin-Mazur zeta function} of $f$ is the formal power series

\[
\zeta_f(t) = \exp \left ( \sum_{i=1}^\infty \frac{|\Fix(f^i)|}{i} \cdot t^i \right ).
\]

Let $f$ be a $\lambda$-zig-zag map with $m$ critical points. We say $f$ is \textit{simple} if for some minimal $N \geq 1$, $f^N(1)$ is a critical point:

\[
\text{$f^N(1) = \frac{k}{\lambda}$ \hspace{5mm} for some $1 \leq k \leq m$.}
\]

\begin{rmk}
Note that each $f \in \PA(m)$ is simple with $k=1$.
\end{rmk}

If $f$ is simple, then the \textit{coefficients} of $f$ are the integers $c_i$ for $0 \leq i \leq N-1$ such that

\[
c_i(f) = \begin{cases}
A(f^i(1)) & \text{if $E(A(f^i(1)))=+1$}\\
A(f^i(1))+1 & \text{if $E(A(f^i(1)))=-1$.}
\end{cases}
\]

\noindent Furthermore, we set $s_i(f) = s_i(\It_f(1))$, the cumulative signs of the itinerary of $f$. Finally, define the polynomials

\[
\rho_f(t) = k \cdot s_N(f) \cdot  t^N + \sum_{i=0}^{N-1} s_i(f) c_i(f) t^i
\]

\noindent and 

\[
\phi_f(t) = t \cdot \rho_f(t).
\]

\begin{prop}\label{prop:digitpolyeq}
Suppose $f \in \PA(m)$. Then 

\[
1 - \phi_f(t) = \calR(D_f)(t).
\]
\end{prop}

\begin{proof}
Write $\Phi(f) = \frac{a}{b}$, and $\lambda = \lambda(f)$. Then $N=b$, with $f^b(1)=\frac{1}{\lambda}$. Furthermore, Remark \ref{rmk:possign} shows that $s_j(f) = +1$ for $0 \leq j \leq b-1$, whereas $s_b(f) = s_b(\It_f(1)) = -1$. Therefore,

\[
1-\phi_f(t) = 1-t \rho_f(t) = 1+t^{b+1} - \sum_{i=1}^b c_{i-1}(f) t^i.
\]

On the other hand, we know that

\[
D_f(t) = t^{b+1} + 1 - \sum_{i=1}^b c_i t^{b+1-i},
\]

\noindent where the $c_i$ satisfy the relations in (\ref{eqn:coefrelations}). Taking the reverse gives 

\[
\calR(D_f)(t) = t^{b+1}+1-\sum_{i=1}^b c_i t^i,
\]

\noindent so it remains to show that $c_i = c_{i-1}(f)$ for $1 \leq i \leq b$. Indeed, for $1 \leq i \leq b-1$ we have

\[
c_i = \nu_{i-1}(f) = A(f^{i-1}(1)) = c_{i-1}(f),
\]

\noindent while Lemma \ref{fundlem:Digit} tells us that

\[
c_b = m = (m-1)+1 = c_{b-1}(f).
\]
\end{proof}

Combining Proposition \ref{prop:digitpolyeq} above and Theorem 1.1 of \cite{S} gives the following:

\begin{prop}\label{prop:BLzeta}
Fix $f \in \PA(m)$. Then $\zeta_f(t)$ converges absolutely in $|t| < \lambda(f)^{-1}$. In addition, for $|t| < \lambda(f)^{-1}$ we have

\begin{equation}\label{eqn:zeta}
\zeta_f(t) = \frac{1}{\calR(D_f(t))} = \frac{1}{\det(I-tW_f)}.
\end{equation}

\noindent In particular,

\begin{equation}\label{eqn:charanddigit}
\chi_W(f; t) = t^{m-1} D_f(t).
\end{equation}
\end{prop}

\begin{proof}
Since $f$ is continuous, Suzuki's argument in the proof of Theorem 1.1 in \cite{S} verifies (\ref{eqn:zeta}). Therefore,

\[
\calR(D_f(t)) = \calR(\det(tI-W_f)) = \calR(\chi_W(f; t)).
\]

\noindent Since $\deg(D_f) = b+1$ and $\dim(W_f) = b+1 + (m-1)$, Proposition \ref{prop:rev_poly} implies formula (\ref{eqn:charanddigit}).
\end{proof}

\begin{cor}\label{cor:chardigit}
Fix $f \in \PA(m)$. Then $\chi_M(f; t) = D_f(t)$.
\end{cor}

\begin{proof}
Proposition \ref{prop:Markov_poly} implies that $\chi_W(f; t) = t^{m-1} \chi_M(f; t)$. The statement now follows from Proposition \ref{prop:BLzeta}.
\end{proof}

\subsection{$D_f$ is the homology polynomial of $\psi_f$}\label{subsec:BBK}

In \cite{BBK}, Birman-Brinkmann-Kawamuro introduce several related polynomial invariants of a pseudo-Anosov $\psi$ on a closed, connected, orientable surface $S$ with punctures, treated as marked points. These are:

\begin{itemize}
\item the \textit{homology polynomial} $h(\psi; t)$, which records the action of $\psi$ on a vector space $W$ of weight functions for an invariant train track $\tau$ which satisfy the \textit{switch conditions}. 
\item the \textit{puncture polynomial} $p(\psi; t)$, which records the action of $\psi$ on the degenerate subspace $Z \subseteq W$ of a certain skew-symmetric bilinear form. This polynomial is a product of cyclotomics, and may be reinterpreted using the permutation action of $\psi$ on a certain subset of the punctures of $S$.
\item the \textit{symplectic polynomial} $s(\psi; t)$, which records the action of $\psi$ on $W/Z$.
\end{itemize}

Note that $h(\psi; t) = s(\psi; t) \cdot p(\psi; t)$. The symplectic polynomial, and hence the homology polynomial, has the stretch factor $\lambda(\psi)$ as its root of maximal modulus.\\

In this section we prove statement (\ref{poly_invts:hom_1}) of Theorem \ref{mainthm:poly_invts}.

\begin{thm}\label{thm:hom_poly}
Fix $f \in \PA(m)$ and $\psi_f$ the pseudo-Anosov it generates. Set $n=|\PC(f)|$. Then

\[
D_f(t) = h(\psi_f; t) = \begin{cases}
s(\psi_f; t) & \text{if $n$ is odd}\\
s(\psi_f; t)(t+1) & \text{if $n$ is even.}
\end{cases}
\]
\end{thm}



\begin{figure}
\centering
\includegraphics[scale=.4]{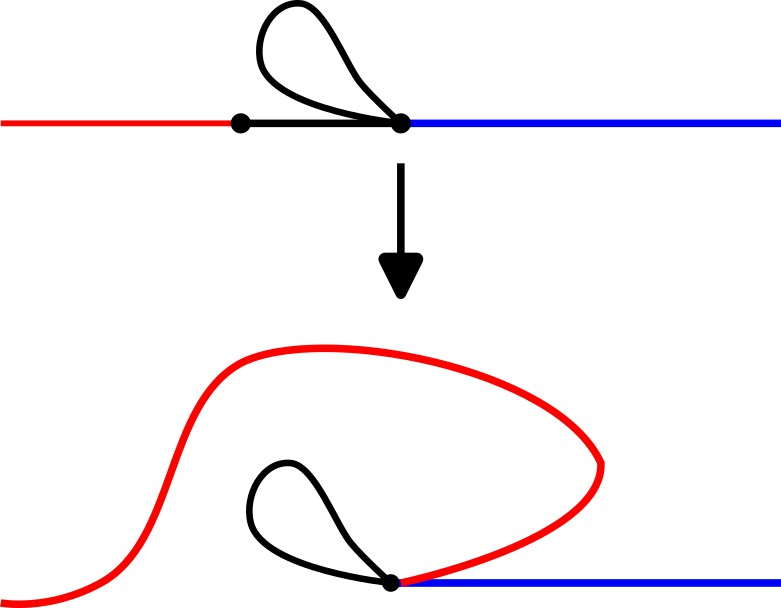}
\caption{The folding move in the proof of Theorem \ref{thm:hom_poly}. The loop and the middle edge are \textit{infinitesimal edges}, while the other two edges are expanding. The two infinitesimal edges have the same image under $f$. Thus we may identify them to obtain a new graph $\tau'$ with one fewer edge. Applying this move to every non-extremal loop in $\tau$ produces a Bestvina-Handel train track that satisfies the requirements of \cite{BBK}. Importantly, the transition matrix of the new train track map is identical to that of the original, when restricting to the expanding edges.}\label{fig:folding}
\end{figure}

\begin{proof}
Recall that the homology polynomial $h(\psi_f; t)$ is equal to the characteristic polynomial of the action of $\psi_f$ on a space of weights on an underlying graph that satisfy the \textit{switch conditions} (cf. Theorem 2.1 in \cite{BBK}.) We first compute the dimension of this weight space, denoted $W(G,f)$, and show that it is equal to $\deg(D_f) = n-1$.

We make minor adjustments to our train track, which is not in quite the same form as that assumed by Birman-Brinkmann-Kawamuro. We refer the reader to Section 5 of \cite{F} for a description of the general form of our train tracks. Recall that each loop $\gamma$ of $\tau$ corresponds to an element of $\PC(f)$. The non-extremal loops of $\tau$ are those that do not correspond to $x=0, 1$ in $I$. For each non-extremal loop, we apply a \textit{folding move} as in Figure \ref{fig:folding}. The resulting train track $\tau'$ is still invariant for $\psi_f$,  and has the same transition matrix $M_f$, but is now a Bestvina-Handel track. In particular, it satisfies the hypotheses in \cite{BBK}.

The underlying graph $G$ of this new track $\tau'$ has $n-1$ expanding edges, $n$ infinitesimal loops, and $n$ vertices. All of the vertices are \textit{partial}, in the terminology of Birman-Brinkmann-Kawamuro, and hence we have the count

\begin{align*}
\deg(h(\psi_f; t)) = \dim{W(G,f)} & = \text{$\#$edges of $G$ - $\#$non-odd vertices of $G$}\\
& = (n-1+n)-n\\
& = n-1.
\end{align*}

To show that $h(\psi_f; t) = D_f(t)$, it remains to find a basis for $W(G, f)$. There is a natural such basis: each expanding edge $e_i$ is adjacent to two infinitesimal loops, $\gamma_{i-1}$ and $\gamma_i$. Define $\eta_i \in W(G, f)$ to be the weight function such that

\[
\eta_i(e_i)=1, \hspace{5mm} \eta_i(\gamma_{i-1})=\eta_i(\gamma_i) = 1/2,
\]

\noindent and $\eta_i=0$ outside of these three edges. The functions $\{\eta_i\}_{i=1}^{n-1}$ are linearly independent, since exactly one of them assigns non-zero weight to any expanding edge. Therefore, these $\eta_i$ form a basis for $W(G,f)$. Moreover, by the construction of the train track the action of $\psi_f$ on the expanding edges of $\tau'$ is precisely the action of $f$ on the subintervals of the strong Markov partition. Thus,

\[
h(\psi_f; t) = \chi_M(f; t) = D_f(t).
\]

Next we compute the puncture polynomial $p(\psi_f; t)$.  Recall that this polynomial is the characteristic polynomial of the action of $\psi_f$ on a certain degenerate subspace $Z$ of $W(G, f)$.  The dimension of $Z$ is equal to the number of punctures in $S_{0, n+1}$ represented by a loop with an even number of corners in $\tau'$ (cf. Proposition 3.3 in \cite{BBK}). The number of corners in such a loop is precisely equal to the number of prongs of the singularity at the puncture. Therefore, the puncture at infinity is the only puncture that could possibly contribute to the dimension count of $Z$, and it does so precisely when it has an even number of prongs, i.e. when $n-2$ is even. Thus,

\[
\deg(p(\psi_f; t)) = \dim(Z) = \begin{cases}
0 & \text{if $n$ is odd}\\
1 & \text{if $n$ is even.}
\end{cases}
\]

\noindent In particular, since $p(\psi_f; t)$ is a product of cyclotomics and since $D_f(1) \neq 0$ we have

\[
p(\psi_f; t) = \begin{cases}
1 & \text{if $n$ is odd}\\
t+1 & \text{if $n$ is even.}
\end{cases}
\]

Therefore, the symplectic polynomial $s(\psi_f; t)$ satisfies

\[
s(\psi_f; t) = \frac{h(\psi_f; t)}{p(\psi_f; t)} = \begin{cases}
D_f(t) & \text{if $n$ is odd}\\
D_f(t)/(t+1) & \text{if $n$ is even.}
\end{cases}
\]
\end{proof}

\subsection{The orientation double cover for $\psi_f$}\label{sec:orientation}

In this section we prove statement (\ref{poly_invts:hom_2}) of Theorem \ref{mainthm:poly_invts}.

A singular foliation $\calF$ of a surface $S$ is \textit{orientable} if there exists a vector field $X$ on $S$ which is zero at the singularities of $\calF$ and everywhere else is nonzero and tangent to $\calF$. Otherwise, we say $\calF$ is \textit{non-orientable}. Similarly, we may speak of a train track $\tau \subseteq S$ being orientable: say $\tau$ is \textit{orientable} if there is a choice of orientation for each edge of $\tau$ such that, if there is a smooth path from $e_1$ to $e_2$ then the orientation of $e_2$ agrees with the orientation of $e_1$ flowed along the path to $e_2$.

Since the literature on this subject is vast and since there is some confusion on the relation between these notions of orientability, we include here a brief summary of relevant properties. 

\begin{prop}\label{prop:orient}
Let $\psi: S \to S$ be a pseudo-Anosov of the compact, connected, oriented surface $S$. Then the following are equivalent:

\begin{enumerate}
\item \label{orient:unstable} The unstable foliation of $\psi$ is orientable.
\item \label{orient:stable} The stable foliation of $\psi$ is orientable.
\item \label{orient:track} Any invariant train track for $\psi$ is orientable.
\item \label{orient:eigenvalues} The spectral radius of the map $\psi_\ast: H_1(S; \bbZ) \to H_1(S; \bbZ)$ is the stretch factor $\lambda$ of $\psi$, and in fact exactly one of $\lambda$ and $-\lambda$ is an eigenvalue of $\psi_\ast$.
\end{enumerate}
\end{prop}

\begin{proof}
The proof of Lemma 4.3 in \cite{BB} provides the equivalences

\[
(\ref{orient:unstable}) \iff (\ref{orient:stable}) \iff (\ref{orient:eigenvalues}).
\]

It remains to show that (\ref{orient:unstable}) is equivalent to (\ref{orient:track}). If the unstable foliation $\calF$ of $\psi$ is orientable, then any invariant train track $\tau$ of $\psi$ carries $\calF$, and in particular inherits an orientation from that of $\calF$. Conversely, an orientable train track $\tau$ carries only orientable invariant foliations (cf. Lemma 2.6 in \cite{Los}).
\end{proof}

\begin{defn}
We say a pseudo-Anosov $\psi$ is \textit{orientable} if it satisfies any of the equivalent conditions of Proposition \ref{prop:orient}.
\end{defn}

Thus the orientability of a pseudo-Anosov is an important invariant. There is a standard construction for producing a pseudo-Anosov from a non-orientable one. This construction is called the \textit{orientation double cover}, and underlies the techniques in \cite{BB}, \cite{BBK}, and \cite{Los}. Briefly, one takes the branched double cover $p: \tilde{S} \to S$, branching along the singularities of $\psi$ that have an odd number of prongs. Denoting by $\iota: \tilde{S} \to \tilde{S}$ the unique non-trivial deck transformation for this cover, we see that $\iota^2 = 1$. Hence, we have the decomposition

\[
H_1(\tilde{S}; \bbZ) = V_+ \oplus V_-,
\]

\noindent where $V_{\pm}$ is the eigenspace of $\iota_\ast$ for the eigenvalue $\pm 1$. The pseudo-Anosov $\psi$ lifts to two maps $\psi_1, \psi_2: \tilde{S} \to \tilde{S}$ that both commute with $\iota$ and satisfy $\psi_1 = \iota \circ \psi_2$. 

The following proposition combines the work of \cite{BB} and \cite{BBK}.
\begin{prop}\label{prop:cover}
Let $\psi: S \to S$ be a non-orientable pseudo-Anosov with stretch factor $\lambda > 1$, and let $p: \tilde{S} \to S$ be the orientation double cover of $\psi$. Let $\iota: \tilde{S} \to \tilde{S}$ be the unique non-trivial deck transformation for this cover. Then:

\begin{enumerate}
\item Both $\psi_1$ and $\psi_2$ are orientable pseudo-Anosovs with stretch factor $\lambda$.
\item \label{cover:ker} The kernel of $p_\ast$ is precisely $\ker(p_\ast) = V_- \subseteq H_1(\tilde{S}; \bbZ)$.
\item The actions of $(\psi_1)_\ast$ and $(\psi_2)_\ast$ on $V_+$ are each conjugate to the action of $\psi_\ast$ on $H_1(S; \bbZ)$.
\item Exactly one of $(\psi_1)_\ast$ and $(\psi_2)_\ast$ has $\lambda$ as an eigenvalue, while the other has $-\lambda$ as an eigenvalue. Denote the corresponding homeomorphisms of $\tilde{S}$ by $\psi_+$ and $\psi_-$, respectively.
\item \label{cover:char} The characteristic polynomials of $(\psi_+)_\ast$ and $(\psi_-)_\ast$ acting on $V_-$ satisfy

\[
\text{$\det(tI-(\psi_+)_\ast |_{V_-}) = h(\psi; t)$ \hspace{5mm} and \hspace{5mm} $\det(tI-(\psi_-)_\ast |_{V_-}) = h(\psi; -t)$.}
\]
\end{enumerate}
\end{prop}

Consider $f \in \PA(m)$ with $n=|\PC(f)|$. The pseudo-Anosov $\psi_f$ is defined on $S=S_{0, n+1}$, and its set of singularities coincides with the puncture set of $S$: each of the points of $\PC(f)$ corresponds to a punctured 1-pronged singularity, and the final puncture is a singularity with $n-2$ prongs. The orientation double cover of $\psi_f$ is the surface

\[
\widetilde{S} \cong \begin{cases}
S_{\frac{n-1}{2}, \ n+1}& \text{if $n$ is odd}\\
S_{\frac{n-2}{2}, \ n+2} & \text{if $n$ is even.}
\end{cases}
\]

\noindent In $\widetilde{S}$ there are $n$ punctured 2-prong singularities lying over the $1$-prong singularities in $S$. The remaining punctures (there are either one or two) project to the final puncture of $S$. In either case, Proposition \ref{prop:cover} implies that 

\begin{equation}\label{eqn:homology}
H_1(\widetilde{S}; \bbZ) \cong H_1(S; \bbZ) \oplus V_- \cong \bbZ^n \oplus \bbZ^{n-1},
\end{equation}

\noindent and that the actions of $(\psi_f)_+$ and $(\psi_f)_-$ preserve this splitting.

We are now ready to prove statement (\ref{poly_invts:hom_2}) of Theorem \ref{mainthm:poly_invts}.

\begin{thm}
Let $\bbS$ be the surface obtained from $\widetilde{S}$ by filling in the lifts of the punctured $1$-prong singularities of $\psi_f$. Denote by $\chi_+(t)$ (resp., $\chi_-(t)$) the characteristic polynomial of $\psi_+$ (resp., $\psi_-$) acting on $H_1(\bbS; \bbZ)$. Then

\[
D_f(t) = \chi_+(t) \hspace{5mm} \text{and} \hspace{5mm} D_f(-t) = \chi_-(t).
\]
\end{thm}

\begin{proof}
We prove that $D_f(t) = \chi_+(t)$, since the other statement follows immediately. Filling in the $2$-prong singularities of $\widetilde{S}$ kills the factor of $H_1(S; \bbZ)$ in (\ref{eqn:homology}). Thus, the action of $(\psi_+)_\ast$ on $H_1(\bbS; \bbZ)$ is conjugate to the action of $(\psi_+)_\ast$ on $V_-$. By statement (\ref{cover:char}) of Proposition \ref{prop:cover} we find

\[
\chi_+(t) = h(\psi; t) = D_f(t),
\]

\noindent where the final equality was proved in Theorem \ref{thm:hom_poly}.
\end{proof}
\subsection{The Burau representation}\label{subsec:Burau} In this section we complete the proof of Theorem \ref{mainthm:poly_invts}.

Fix $n \geq 2$, and denote by $B_n$ the $n$-stranded braid group. The \textit{reduced Burau representation in dimension $n$} is a homomorphism 

\[
\bbB(\cdot, z): B_n \to \GL(n-1, \bbZ[z, z^{-1}]).
\]

\noindent Thus, the image $\bbB(\beta, z)$ of an $n$-braid $\beta$ is an $(n-1) \times (n-1)$ matrix whose entries are Laurent polynomials in the variable $z$. There are explicit formulas for the Burau representation in terms of the Artin generators for $B_n$; cf. for example Chapter 3 of \cite{Bir}.

\begin{defn}
The \textit{symplectic representation} is the homomorphism

\[
\bbB(\cdot, -1): B_n \to \GL(n-1, \bbZ)
\]

\noindent obtained by specializing $\bbB(\cdot, z)$ at $z=-1$. We define the \textit{symplectic Burau polynomial} of $\beta \in B_n$ to be the characteristic polynomial of $\bbB(\beta, -1)$:

\[
\chi(\beta; t) = \det(tI - \bbB(\beta, -1)).
\]

\noindent Note that $\chi(\beta; t)$ has degree $n-1$.
\end{defn}

It is well-known that the mapping class group of the $n$-punctured disc $D_n$, relative to its boundary circle $\partial D_n$, is

\[
\Mod(D_n, \partial D_n) \cong B_n.
\]

\noindent Moreover, the center of $B_n$ is generated by the \textit{full twist} $\Delta_n^2$, the Dehn twist around $\partial D_n$.  Capping off this boundary component by a punctured disc kills this element of $B_n$, and we obtain the isomorphism

\[
B_n / \Delta_n^2 \cong \Mod(S_{0,n+1},p),
\]

\noindent where the group on the right is the subgroup of $\Mod(S_{0, n+1})$ fixing the puncture $p$ in the capping disc. Thus, given $\psi \in \Mod(S_{0, n+1}, p)$ we may associate to it a braid $\beta \in B_n$, which is well-defined up to multiplication by a power of $\Delta_n^2$.\\

\begin{rmk}\label{rmk:sign}
One can show that $\bbB(\Delta_n^2, -1) = -I_{n-1}$. Therefore, 

\[
\chi(\Delta_n^2 \beta; t) = \chi(\beta; -t).
\]

\noindent In particular, the symplectic Burau polynomial of a braid representative of a pseudo-Anosov $\psi \in \Mod(S_{0, n+1}, p)$ is nearly an invariant of $\psi$ itself: the only ambiguity is a possible change of variable $t \mapsto -t$.
\end{rmk}

Given $f \in \PA(m)$ with $|\PC(f)|=n$ the pseudo-Ansov $\psi_f: S_{0,n+1} \to S_{0, n+1}$ has a unique singularity $\infty$ with $n-2$ prongs, which is necessarily a fixed point. Ripping open this puncture to a boundary component, we obtain an $n$-braid $\beta_f$ defined up to multiplication by $\Delta_n^2$. Regardless of our choice for $\beta_f$, however, it will always be pseudo-Anosov with a 1-prong singularity at each puncture of $D_n$ and an $(n-2)$-prong singularity at $\partial D_n$. Moreover, the stretch factor of $\beta_f$ is equal to the stretch factor $\lambda$ of $\psi_f$. 

With the remark in mind, we state the main result of this section.

\begin{thm}\label{thm:burau_poly}
Fix $f \in \PA(m)$, and let $\beta_f$ be a braid representative of $\psi_f$ after ripping open $\infty \in S_{0,n+1}$ to a boundary circle. Then

\[
\chi(\beta_f; t) = \begin{cases}
D_f(t) & \text{if $\chi(\beta_f; t)$ has $\lambda$ as a root}\\
D_f(-t) & \text{if $\chi(\beta_f; t)$ has $-\lambda$ as a root.}
\end{cases}
\]
\end{thm}

Our argument relies heavily on the work of Band-Boyland in \cite{BB}. We direct the reader to sections 2 and 3 of that paper for details. The reduced Burau matrix $\bbB(\beta, z)$ describes the action on first homology of a preferred lift $\tilde{h}$ of $\beta$ to a certain cover of $D_n$, denoted $D_n^{(\infty)}$. The deck group of the cover $\tilde{p}: D_n^{(\infty)} \to D_n$ is isomorphic to $\bbZ$, and thus for each $k \geq 1$ we obtain a cover $p_k: D_n^{(k)} \to D_n$ via the quotient map $\xi_k: \bbZ \to \bbZ/k\bbZ$, as well as a preferred lift $h^{(k)}: D_n^{(k)} \to D_n^{(k)}$ of $\beta$. 

Importantly, one obtains information about $h_\ast^{(k)}: H_1(D_n^{(k)}) \to H_1(D_n^{(k)})$ by specializing $\bbB(\beta, z)$ at the $k$th roots of unity. 

\begin{prop}[Theorem 3.4 in \cite{BB}]\label{prop:eigen_decomp}
For $k \geq 1$ set $\zeta_k$ to be a primitive $k$th root of unity. There is an invariant subspace $S_\bbC^{(k)}$ of $H_1(D_n^{(k)}; \bbC)$ such that the action of $h^{(k)}_\ast$ on this subspace is by

\begin{equation}\label{eqn:eigen_decomp}
\bbB(\beta, 1) \oplus \bbB(\beta, \zeta_k) \oplus \cdots \oplus \bbB(\beta, \zeta_k^{k-1}).
\end{equation}

\noindent Moreover, set $T$ to be the generator of the deck group $\bbZ/k\bbZ$ for $p_k: D_n^{(k)} \to D_n$. Then in (\ref{eqn:eigen_decomp}), the matrix $\bbB(\beta, \zeta_k^j)$ is the action of $h_\ast^{(k)}$ on the eigenspace of $T$ with eigenvalue $\zeta_k^j$.

Finally, any eigenvector of $h_\ast^{(k)}$ not lying in $S_\bbC^{(k)}$ has as its eigenvalue a root of unity.
\end{prop}

Recall that, by Proposition \ref{prop:orient}, a pseudo-Anosov $\psi$ with stretch factor $\lambda$ is orientable if and only if $\lambda$ or $-\lambda$ is an eigenvalue of $\psi_\ast$. Band-Boyland use this fact to conclude the following.

\begin{thm}[Theorem 5.1 in \cite{BB}]\label{thm:BB}
Suppose $\beta \in B_n$ is a pseudo-Anosov braid having stretch factor $\lambda > 1$. Then the following are equivalent:

\begin{enumerate}
\item The spectral radius of $\bbB(\beta, -1)$ is $\lambda$, and $-1$ is the only root of unity for which this occurs.
\item \label{BB:foliations} The invariant foliations $\calF^u$ and $\calF^s$ have an odd-order singularity at each puncture of $D_n$ (not including the boundary), and all other singularities of $\calF^u$ and $\calF^s$ in the interior of $D_n$ have even order.
\item $D_n^{(2)}$ is the orientation double cover of $\beta$.
\end{enumerate}
\end{thm}

We are now ready to prove Theorem \ref{thm:burau_poly}. This will conclude the proof of Theorem \ref{mainthm:poly_invts}.

\begin{proof}[Proof of Theorem \ref{thm:burau_poly}]
Let $\beta_f$ be any pseudo-Anosov braid obtained from $\psi_f$ by ripping open $\infty \in S_{0,n+1}$ to the boundary circle of $D_n$. Then $\beta_f$ has the following singularity data:

\begin{itemize}
\item a $1$-prong singularity at each puncture of $D_n$, and
\item a singularity with $n-2$ prongs on $\partial D_n$.
\end{itemize}

\noindent Therefore, $\beta_f$ satisfies condition (\ref{BB:foliations}) of Theorem \ref{thm:BB}. It follows that $D_n^{(2)}$ is the orientation double cover of $\beta_f$. In this setting, the deck group of $p_2: D_n^{(2)} \to D_n$ is generated by $T=\iota$, the hyperelliptic involution of $D_n^{(2)}$.

On the other hand, one may take the orientation cover of $\psi_f: S_{0, n+1} \to S_{0, n+1}$, obtaining the surface $\widetilde{S}$ with its own hyperelliptic involution $\iota'$. As we have noted before, $\psi_f$ has a pair of lifts, $\psi_+$ and $\psi_-=\iota' \circ \psi_+$. \\

We will relate the following two actions:

\begin{enumerate}
\item the action of the preferred lift $\beta_f^{(2)}$ on $V_-(\iota)$
\item the action of a lift $\widetilde{\psi}$ of $\psi$ on $V_-(\iota')$.
\end{enumerate}

\noindent From Proposition \ref{prop:eigen_decomp} we know that action (1) has characteristic polynomial

\[
\det \left (tI - \left ( \beta_f^{(2)} \right )_\ast \bigg |_{V_-(\iota)} \right ) = \det(tI-\bbB(\beta_f, -1)) = \chi(\beta_f; t).
\]

\noindent For action (2), the work of Birman-Brinkmann-Kawamuro tells us that

\[
\det \left ( tI - \left ( \widetilde{\psi} \right )_\ast \bigg |_{V_-(\iota')} \right) = \begin{cases}
h(\psi_f; t) & \text{if $\widetilde{\psi} = \psi_+$}\\
h(\psi_f; -t) & \text{if $\widetilde{\psi} = \psi_-$}.
\end{cases}
\]

There is a natural isomorphism $\pi: H_1(D_n^{(2)}; \bbZ) \to H_1(\widetilde{S}; \bbZ)$ obtained by capping off each boundary component of $D_n^{(2)}$ with a punctured disc. Moreover, the following diagram commutes:

\[
\begin{tikzcd}[sep=large]
H_1(D_n^{(2)}; \bbZ) \arrow{r}{\pi} & H_1(\widetilde{S}; \bbZ) \\
V_-(\iota) \arrow[hookrightarrow]{u} \arrow{r}{\pi} & V_-(\iota') \arrow[hookrightarrow]{u} \\
V_-(\iota) \arrow{r}{\pi}  \arrow{u}{(\beta_f^{(2)})_\ast} & V_-(\iota') \arrow{u}{\widetilde{\psi}_\ast}
\end{tikzcd}
\]

\noindent Here, $\widetilde{\psi}$ is the unique lift of $\psi_f$ that we obtain by capping off $\beta_f^{(2)}$. We now see that actions (1) and (2) are conjugate, and hence their characteristic polynomials coincide. To finish the proof, we apply Theorem \ref{thm:hom_poly} to conclude that

\[
\chi(\beta_f; t) = \begin{cases}
D_f(t) & \text{if $\widetilde{\psi}=\psi_+$}\\
D_f(-t) & \text{if $\widetilde{\psi}=\psi_-$.}
\end{cases}
\]
\end{proof}

\begin{rmk}
Theorem \ref{thm:burau_poly} is a special case of a result proved in the thesis of Warren Michael Shultz, cf. Theorem 1.0.1 in \cite{Sh}. There is a slight error in the statement of this theorem: namely, item (1) does not include the possible sign ambiguity that we noted in Remark \ref{rmk:sign}.
\end{rmk}

\section{Future directions}

Over the course of \cite{F} and this paper, we have undertaken the study of a particular family $\PA(m)$ of interval maps that may be realized as the train track map for a pseudo-Anosov. In our case, $\PA(m)$ is a one-parameter family, determined by a single critical orbit. Moreover, we have shown that this family exhibits the structure of the Farey tree $\calF$, in that the critical orbit for the map associated to $q_1 \oplus q_2$ is essentially the concatenation of the critical orbits for the maps associated to $q_1$ and $q_2$.

\subsection{Tree-like families}

We strongly suspect that analogous results hold in the following more general context. Let $f: I \to I$ be a PCF interval map of pseudo-Anosov type.  Since $\psi$ permutes the set of $1$-prong singularities, so too does $f$ permute the elements of $\PC(f)$. We conjecture that, fixing all but one postcritical orbit and allowing the final one to vary, a tree-like family of interval maps of pseudo-Anosov type will result. The branching is due to the different choices of how to modify the orbit. In the case of $\PA(m)$, our only choices are to perturb the orbit of $x=1$ to the right or to the left, determining whether we descend to the left or to the right in $\calF$, respectively (cf. Remark \ref{rmk:perturb}).


It is conceivable that some generalization of the invariant $\Phi(f)$ will allow one to track their progress through this tree-like family of pseudo-Anosovs, but the details of such behavior are unclear to us. For one thing, it seems possible that the same rotation number might appear for different maps, unlike what happens for $\PA(m)$.

\subsection{Generalizing Theorem \ref{mainthm:poly_invts}}

What is more clear is how to generalize the statements in Theorem \ref{mainthm:poly_invts} to any interval map of pseudo-Anosov type. For a general such map, the natural replacement for $D_f$ is the strong Markov polynomial $\chi_M(f; t)$. Because the only critical values of a zig-zag map are $x=0$ and $x=1$, and because $x=0$ is either a fixed point or in the orbit of $x=1$, we can extract $\chi_M(f; t)$ purely from the itinerary of $x=1$. Thus we cannot define $D_f(t)$ in general.

An analysis of the proof of Theorem \ref{mainthm:poly_invts} shows, however, that apart from proving that $D_f(t) = \chi_M(f; t)$, the other statements do not rely on any special characteristic of zig-zag maps. The essential fact is that $\chi_M(f;t)$ is the characteristic polynomial of a train track map for $\psi_f$. Statement \ref{poly_invts:hom_1} has to be slightly altered, since it is possible for a factor of $t-1$ to appear if $n=\deg(\chi_M(f;t))+1$ is odd for general $f$, but statements \ref{poly_invts:zeta}, \ref{poly_invts:hom_2}, and \ref{poly_invts:Burau} hold verbatim after replacing $D_f(t)$ with $\chi_M(f; t)$.

\subsection{Interval maps from pseudo-Anosovs}

A PCF interval map $f$ of pseudo-Anosov type with $|\PC(f)| = p$ generates a pseudo-Anosov $\psi_f: S_{0,p+1} \to S_{0,p+1}$ with at most one singularity having two or more prongs. Call a pseudo-Anosov with only one such singularity \textit{hyperpolar}, since each $1$-prong singularity is a \textit{pole} for the associated quadratic differential.

Furthermore, call a pseudo-Anosov \textit{interval-like} if it admits an invariant train track $\tau$ consisting of a chain of expanding edges with an infinitesimal loop at each vertex. Thus, every interval-like pseudo-Anosov is hyperpolar. As we have seen, the dynamics of the PCF interval map underlying $f: \tau \to \tau$ determine those of the interval-like pseudo-Anosov. It is natural, then, to ask the following question.

\begin{question}
Which hyperpolar pseudo-Anosovs are interval-like?
\end{question}

In a forthcoming paper with Karl Winsor, we prove that the two sets of pseudo-Anosovs coincide.

\begin{thm}
Every hyperpolar pseudo-Anosov is interval-like.
\end{thm}

Thus, hyperpolar pseudo-Anosovs may be said to have the dynamics of PCF interval maps. In order to truly leverage this connection, however, it would be useful to have a better understanding of when a PCF interval map is of pseudo-Anosov type. In \cite{F} we achieved this explicitly for the case of zig-zag maps from combinatorial data. In general, one would hope for a set of criteria relying on the kneading data of a given interval map to determine if it is of pseudo-Anosov type.

\begin{question}
Given a PCF interval map $f$, is there a set of criteria, depending on the kneading data of $f$, that is equivalent to $f$ being of pseudo-Anosov type?
\end{question}

\bibliography{Kneading}

\begin{thebibliography}{BBK12}

\bibitem[BB07]{BB}
Gavin Band and Philip Boyland.
\newblock The {B}urau estimate for the entropy of a braid.
\newblock {\em Algebr. Geom. Topol.}, 7:1345--1378, 2007.

\bibitem[BBK12]{BBK}
Joan Birman, Peter Brinkmann, and Keiko Kawamuro.
\newblock Polynomial invariants of pseudo-{A}nosov maps.
\newblock {\em J. Topol. Anal.}, 4(1):13--47, 2012.

\bibitem[BH95]{BH}
M.~Bestvina and M.~Handel.
\newblock Train-tracks for surface homeomorphisms.
\newblock {\em Topology}, 34(1):109--140, 1995.

\bibitem[Bir75]{Bir}
Joan~S. Birman.
\newblock {\em Erratum: ``{\it {B}raids, links, and mapping class groups}''
  ({A}nn. of {M}ath. {S}tudies, {N}o. 82, {P}rinceton {U}niv. {P}ress,
  {P}rinceton, {N}. {J}., 1974)}.
\newblock Princeton University Press, Princeton, N. J.; University of Tokyo
  Press, Tokyo, 1975.
\newblock Based on lecture notes by James Cannon.

\bibitem[BvS15]{BvS}
Henk Bruin and Sebastian van Strien.
\newblock Monotonicity of entropy for real multimodal maps.
\newblock {\em J. Amer. Math. Soc.}, 28(1):1--61, 2015.

\bibitem[dC05]{dC}
Andr\'{e} de~Carvalho.
\newblock Extensions, quotients and generalized pseudo-{A}nosov maps.
\newblock In {\em Graphs and patterns in mathematics and theoretical physics},
  volume~73 of {\em Proc. Sympos. Pure Math.}, pages 315--338. Amer. Math.
  Soc., Providence, RI, 2005.

\bibitem[dCH04]{dCH}
Andr\'{e} de~Carvalho and Toby Hall.
\newblock Unimodal generalized pseudo-{A}nosov maps.
\newblock {\em Geom. Topol.}, 8:1127--1188, 2004.

\bibitem[Far21]{F}
Ethan Farber.
\newblock Constructing pseudo-{A}nosovs from expanding interval maps.
\newblock {\em arXiv:2101.01721}, 2021.

\bibitem[Hal94]{H}
Toby Hall.
\newblock The creation of horseshoes.
\newblock {\em Nonlinearity}, 7(3):861--924, 1994.

\bibitem[Hat22]{Ha}
Allen Hatcher.
\newblock {\em Topology of Numbers}.
\newblock American Mathematical Society, Providence, Rhode Island, 2022.

\bibitem[Khi97]{K}
A.~Ya. Khinchin.
\newblock {\em Continued fractions}.
\newblock Dover Publications, Inc., Mineola, NY, russian edition, 1997.
\newblock With a preface by B. V. Gnedenko, Reprint of the 1964 translation.

\bibitem[Los10]{Los}
J\'{e}r\^{o}me Los.
\newblock Infinite sequence of fixed-point free pseudo-{A}nosov homeomorphisms.
\newblock {\em Ergodic Theory Dynam. Systems}, 30(6):1739--1755, 2010.

\bibitem[MS80]{MS}
M.~Misiurewicz and W.~Szlenk.
\newblock Entropy of piecewise monotone mappings.
\newblock {\em Studia Math.}, 67(1):45--63, 1980.

\bibitem[MT00]{MTr}
John Milnor and Charles Tresser.
\newblock On entropy and monotonicity for real cubic maps.
\newblock {\em Comm. Math. Phys.}, 209(1):123--178, 2000.
\newblock With an appendix by Adrien Douady and Pierrette Sentenac.

\bibitem[Shu21]{Sh}
Warren~Michael Shultz.
\newblock {\em The {H}omology {P}olynomial and the {B}urau {R}epresentation for
  {P}seudo-{A}nosov {B}raids}.
\newblock ProQuest LLC, Ann Arbor, MI, 2021.
\newblock Thesis (Ph.D.)--Michigan State University.

\bibitem[Suz17]{S}
Shintaro Suzuki.
\newblock Artin-{M}azur zeta functions of generalized beta-transformations.
\newblock {\em Kyushu J. Math.}, 71(1):85--103, 2017.

\bibitem[Thu88]{Th3}
William~P. Thurston.
\newblock On the geometry and dynamics of diffeomorphisms of surfaces.
\newblock {\em Bull. Amer. Math. Soc. (N.S.)}, 19(2):417--431, 1988.

\bibitem[Thu14]{Th2}
William~P. Thurston.
\newblock Entropy in dimension one.
\newblock In {\em Frontiers in complex dynamics}, volume~51 of {\em Princeton
  Math. Ser.}, pages 339--384. Princeton Univ. Press, Princeton, NJ, 2014.

\end{thebibliography}
\bibliographystyle{alpha}
\end{document}